\documentclass[12pt]{article}
\usepackage{geometry}                % See geometry.pdf to learn the layout options. There are lots.
\usepackage{graphicx,color,mathtools}
\usepackage{amssymb,amsmath,amsthm,mathrsfs}
\usepackage[all,cmtip]{xy}
\usepackage{epstopdf, comment, url, esint}
\usepackage{bm} % Bold math symbols
\usepackage{enumitem}

\numberwithin{equation}{section}

\newtheorem{theorem}{Theorem}[section]

\newtheorem{lemma}[theorem]{Lemma}
\newtheorem{corollary}[theorem]{Corollary}

\theoremstyle{remark}
\newtheorem*{remark}{Remark}
\newtheorem*{example}{Example}

\theoremstyle{definition}

\DeclareMathOperator{\aut}{Aut}

\DeclareMathOperator{\hyp}{hyp}

\DeclareMathOperator{\mob}{mob}
\DeclareMathOperator{\lin}{lin}

\DeclareMathOperator{\re}{Re}
\DeclareMathOperator{\im}{Im}
\DeclareMathOperator{\rad}{rad}
\DeclareMathOperator{\rough}{rough}
\DeclareMathOperator{\smooth}{smooth}

\DeclareMathOperator{\thick}{thick}
\DeclareMathOperator{\thin}{thin}
\DeclareMathOperator{\good}{\varepsilon-good}

\DeclareMathOperator{\Mgood}{\varepsilon-M.good}
\DeclareMathOperator{\Mnice}{\varepsilon-M.nice}

\DeclareMathOperator{\Lgood}{\varepsilon-L.good}
\DeclareMathOperator{\Lnice}{\varepsilon-L.nice}

\DeclareMathOperator{\leaf}{leaf}
\DeclareMathOperator{\gen}{gen}

\DeclareMathOperator{\nice}{nice}

\DeclareMathOperator{\sph}{\hat{\mathbb{C}}}
\DeclareMathOperator{\signed}{s}
\DeclareMathOperator{\Euc}{Euc}

\def\({\big(} \def\){\big)}
\def\N{\mathbb N}
\DeclareMathOperator{\err}{Err}

\title{\bf \Large{Inner Functions and Laminations}}
\author{Oleg Ivrii and Mariusz Urba\'nski}

\date{}                                           % Activate to display a given date or no date

\begin{document}

\maketitle

\abstract{ In this paper, we study orbit counting problems for inner functions using geodesic and horocyclic flows on Riemann surface laminations.  For a one component inner function of finite Lyapunov exponent with $F(0) = 0$, other than $z \to z^d$, we show that the number of pre-images of a point $z \in \mathbb{D} \setminus \{ 0\}$ that lie in a ball of hyperbolic radius $R$ centered at the origin satisfies
$$
\mathcal{N}(z, R) \, \sim \, \frac{1}{2} \log \frac{1}{|z|} \cdot \frac{1}{\int_{\partial \mathbb{D}} \log |F'| dm}, \quad \text{as }R \to \infty.
$$
For a general inner function of finite Lyapunov exponent, we show that the above formula holds up to a Ces\`aro average. Our main insight is that iteration along almost every inverse orbit is asymptotically linear. We also prove analogues of these results for parabolic inner functions of infinite height.
}

\tableofcontents

\section{Introduction}

A {\em finite Blaschke product} $F(z)$ is a holomorphic self-map of the unit disk which extends to a continuous dynamical system on the unit circle. Loosely speaking, an {\em inner function} is a holomorphic self-map of the unit disk which extends to a measure-theoretic dynamical system of the unit circle. More precisely, we require that
for a.e.~$\theta \in [0, 2\pi)$, the radial boundary value $F(e^{i\theta}) := \lim_{r \to 1} F(re^{i\theta})$ exists and has absolute value 1.

 If the Denjoy-Wolff point of $F$ is in the unit disk, then without loss of generality we may assume that $F(0) = 0$, so that 0 is an attracting fixed point of $F$ and the normalized Lebesgue measure $m = |d\theta|/2\pi$ is invariant under $F$. (In this case,  we say that $F$ is {\em centered}\/.)

Let $z \in \mathbb{D} \setminus \{ 0 \}$ be a point on the unit disk, other than the origin. For $R > 0$, we may count the number of repeated pre-images $w$ which lie in the ball of hyperbolic radius $R$ centered at the origin:
$$
\mathcal N(z, R) = \# \bigl \{ w \in B_{\hyp} (0,R) : F^{\circ n}(w) = z \text{ for some }n \ge 0 \bigr \}.
$$

Our first main theorem states:
\begin{theorem}
\label{main-thm}
Let $F$ be an inner function of finite Lyapunov exponent
$$
\chi_m = \int_{\partial \mathbb{D}} \log |F'(re^{i\theta})| dm < \infty,
$$
with $F(0) = 0$ which is not a rotation. If $z \in \mathbb{D} \setminus \{ 0 \}$ lies outside a set of Lebesgue zero measure, then
\begin{equation}
\label{eq:main-thm}
\lim_{R\to+\infty}
\frac{1}{R} \int_0^R \frac{\mathcal N(z, S)}{e^S} dS 
=\frac{1}{2} \log \frac{1}{|z|} \cdot \frac{1}{\int_{\partial \mathbb{D}} \log |F'| dm}.
\end{equation}
\end{theorem}

According to the original definition of W.~Cohn in \cite{cohn}, an inner function $F(z)$ is a {\em one component inner function} if the set $\{ z \in \mathbb{D} : |F(z)| < \rho \}$ is connected for some $0 < \rho < 1$. For applications to dynamical systems, it is more useful to say that an inner function 
is a one component inner function if the set of singular values is compactly contained in the unit disk. This implies that backward iteration along every inverse orbit is asymptotically linear. 

Our second main theorem states:

\begin{theorem}
\label{main-thm2}
Let $F$ be a one component inner function of finite Lyapunov exponent with $F(0) = 0$, other than $z \to z^d$ for some $d \ge 2$. Suppose $z \in \mathbb{D} \setminus \{ 0 \}$ lies outside a set of countable set. Then,
\begin{equation}
\label{eq:main-thm2}
\mathcal N(z, R) \, \sim \, \frac{1}{2} \log \frac{1}{|z|} \cdot \frac{1}{\int_{\partial \mathbb{D}} \log |F'| dm} \cdot e^R,
\end{equation}
as $R \to \infty$.
\end{theorem}

We also obtain analogous results for finite Lyapunov exponent parabolic inner functions of infinite height (in this case, the Denjoy-Wolff point lies on the unit circle). Precise statements will be given in Part \ref{part:parabolic} of the paper.

\begin{remark}
(i) Theorems \ref{main-thm} and \ref{main-thm2} may not hold for every point $z \in \mathbb{D}$. For instance, the inner function
$$
f(z) = \exp \biggl ( \frac{z+1}{z-1} \biggr )
$$
omits the value 0. Post-composing with a M\"obius transformation, we get an inner function $F$ with $F(0) = 0$ which omits a value $p \ne 0$. For $z = p$, the set of repeated pre-images of $z$ is empty.

(ii) For $z \to z^d$, $d \ge 2$, repeated pre-images of a point come in packets, so $\mathcal N(z, R)$ is a step function.

(iii) For an alternative approach to orbit counting using thermodynamic formalism, see \cite[Section 7]{ivrii-wp} and \cite{inner-tdf}. The results in this paper are somewhat stronger because they only require the minimal hypotheses on the inner function $F$; however, the techniques are specific to inner functions.  

(iv) For an analytic characterization of inner functions of finite Lyapunov exponent, we refer the reader to the works \cite{inner, stable, IK22}.
\end{remark}

\subsection{An overview of the proofs}

To prove Theorems \ref{main-thm} and \ref{main-thm2}, we study the geodesic flow on the Riemann surface lamination $\widehat{X}_F$ associated to $F$, which was described in \cite{mcmullen} for finite Blaschke products. (Definitions will be given in Section \ref{sec:fbp}.)
McMullen's construction generalizes to one component inner functions without much difficulty. 
According to Sullivan's dictionary, the Riemann surface lamination is analogous to the unit tangent bundle of a Riemann surface. McMullen showed that the geodesic flow on $\widehat{X}_F$ is ergodic by relating it to a suspension flow over the solenoid. Applying the ergodic theorem to a particular function on the lamination shows Theorem  \ref{main-thm2} up to taking a Ces\`aro average.

To give a full proof of Theorem \ref{main-thm2}, one needs to show that the geodesic flow on $\widehat{X}_F$ is mixing. As in the case of the geodesic flow on a finite area hyperbolic surface, this is done by first showing that the horocyclic flow is ergodic. The main step is to show that the horocyclic flow on $\widehat{X}_F$ has a dense orbit. This uses an argument of A.~Glutsyuk \cite{glutsyuk} which involves examining horocycles on a special leaf of $\widehat{X}_F$ associated to a repelling fixed point on the circle. From here, the ergodicity of the horocyclic flow follows from an argument of Y.~Coud\`ene \cite{coudene}.

Theorem \ref{main-thm} requires more work because one has to manually construct the natural volume form $d\xi$ and the geodesic flow $g_t$ on the lamination  $\widehat{X}_F$ for a general inner function $F$ of finite Lyapunov exponent. To do this, we first show that iteration along almost every inverse orbit is asymptotically linear. The proof uses a number of concepts from differential geometry such as Gaussian and geodesic curvatures. 

\begin{remark}
 In \cite[Section 10]{mcmullen-rtree}, one learns that inner functions are close to hyperbolic isometries away from the critical points. Consequently, a generic inverse orbit stays away from the critical points. 
 \end{remark}

\section{Inner functions}
\label{sec:inner}

As is well known, any inner function $F$ can be factored into a Blaschke product and a singular inner function:
$$
B(z) = e^{i\theta} \prod -\frac{a_i}{|a_i|} \cdot \frac{z-a_i}{1-\overline{a_i}z}, \qquad a_i \in \mathbb{D},
$$
$$
S(z) = \exp \biggl ( - \int_{\partial \mathbb{D}} \frac{\zeta+z}{\zeta-z}\, d\mu(\zeta) \biggr ), \qquad \mu \ge 0, \quad \mu \perp m.
$$
In this decomposition, the Blaschke product records the zero set of $F$, while the singular factor records the zeros of $F$ ``dissolved'' on the unit circle.

The above decomposition privileges the set of pre-images of 0. To view an inner function from the perspective of a point $a \in \mathbb{D}$, we consider the {\em Frostman shift}
$$
F_a(z) = \frac{F(z) - a}{1 - \overline{a} F(z)}.
$$
A point $a \in \mathbb{D}$ is called {\em exceptional}\/ if $F_a$ has a non-trivial singular factor.  Frostman showed that the set of exceptional points in the unit disk has logarithmic capacity 0, while Ahern and Clark \cite{ahern-clark} observed that for inner functions of finite Lyapunov exponent, the exceptional set is at most countable.

The following identity will play an important role in this work:

\begin{lemma}
\label{sum-of-heights}
Suppose $F$ is an inner function with $F(0) = 0$. For a {\em non-exceptional} point $z \in \mathbb{D} \setminus \{ 0 \}$,
\begin{equation}
\sum_{F(w) = z} \log \frac{1}{|w|} = \log \frac{1}{|z|}.
\end{equation}
The $\le$ inequality holds for every $z \in \mathbb{D}$.
\end{lemma}

A proof can be found in \cite[Lemma A.4]{stable}. A holomorphic self-map of the unit disk $F$ has an {\em angular derivative} in the sense of Carath\'eodory
at $\zeta \in \partial \mathbb{D}$ if 
$$F(\zeta) := \lim_{r \to 1} F(r \zeta) \in \partial \mathbb{D} \qquad \text{and} \qquad F'(\zeta) := \lim_{r \to 1} F'(r \zeta) < \infty.
$$
We will use the following two lemmas on angular derivatives from \cite{ahern-clark}:

\begin{lemma}
\label{derivative-circle}
If we decompose $F = BS_\mu$ into a Blaschke product with zero set $\{a_i\}$ and a singular inner function with singular measure $\mu$, then
$$
|F'(\zeta)| = \sum \frac{1-|a_i|^2}{|\zeta - a_i|^2} + \int_{\partial \mathbb{D}} \frac{2d\mu(z)}{|\zeta - z|^2}, \qquad \zeta \in \partial \mathbb{D}.
$$
\end{lemma}
In particular, if $F(0) = 0$ and $F$ is not a rotation, then $|F'(\zeta)| > c > 1$.
\begin{lemma}
\label{ac-lemma}
If an inner function $F$ has an angular derivative at $\zeta \in \partial \mathbb{D}$, then
\begin{equation}
|F'(r\zeta)| \le 4 |F'(\zeta)|, \qquad 0 < r < 1.
\end{equation}
\end{lemma}

The following lemma is a simple consequence of the Schwarz lemma and the triangle inequality:

\begin{lemma}
\label{minimal-translation}
Suppose $F$ is an inner function with $F(0) = 0$, which is not a rotation. There exists a number $\gamma = \gamma(F) > 0$ so that for any $z \in \mathbb{D}$ with $d_{\mathbb{D}}(0, z) \ge 1$, the hyperbolic distance
$$
d_{\mathbb{D}}(0, f(z)) \le d_{\mathbb{D}}(0, z) - 4 \gamma. 
$$
\end{lemma}

The above lemma shows that any ball
$\mathscr B$ of hyperbolic radius $\gamma$ contained in $\{ w \in \mathbb{D} : d_{\mathbb{D}}(0,w) \ge 1 \}$ does not intersect its forward orbit, i.e.~$F^{\circ n}(\mathscr B) \cap \mathscr B = \emptyset$, which implies that its inverse images $\{ F^{-n}(\mathscr B) \}$ are disjoint.

\begin{lemma}
\label{orbit-counting-a-priori}
Let $F(z)$ be an inner function with $F(0) = 0$ that is not a rotation.
For a point  $z \in \mathbb{D}$ in the unit disk with $d_{\mathbb{D}}(0,z) > 1$, we have:
 \begin{equation}
 \label{eq:orbit-counting-a-priori}
 \mathcal N(z, R-1, R) \, := \, \mathcal N(z, R) - \mathcal N(z, R-1) \, \le \, C e^{R - d_{\mathbb{D}}(0,z)}.
\end{equation}
In particular,
\begin{equation}
 \label{eq:orbit-counting-a-priori2}
 \mathcal N(z, R) \le C e^{R - d_{\mathbb{D}}(0,z)},
 \end{equation}
 albeit with a slightly larger constant $C$.
\end{lemma}

\begin{proof}
Since $F$ is not a rotation, by Lemma \ref{minimal-translation}, 
\begin{equation}
\label{eq:decrease-by-c}
d_{\mathbb{D}}(0, F(w)) \le d_{\mathbb{D}}(0,w) - \gamma,
\end{equation}
for any $w \in \mathbb{D}$ with $d_{\mathbb{D}}(0,w) \ge 1$.
Repeated use of Lemma \ref{sum-of-heights} shows that for any $R \ge 1$,
\begin{equation}
\label{eq:w-and-z}
 \sum \log \frac{1}{|w|} \le \log \frac{1}{|z|},
\end{equation}
where the sum is over $\mathcal N(z, R-\gamma,R)$ repeated pre-images $w$ of $z$ for which
$$
R - \gamma \, \le \, d_{\mathbb{D}}(0, w) \, < \, R.
$$
In terms of hyperbolic distance from the origin, (\ref{eq:w-and-z}) says that
$$
\mathcal N(z, R-\gamma,R) \cdot e^{-R} \lesssim e^{-d_{\mathbb{D}}(0,z)},
$$
which shows (\ref{eq:orbit-counting-a-priori}) with  $\mathcal N(z, R-\gamma,R)$  in place of  $\mathcal N(z, R-1,R)$. To obtain the original statement, one just needs to partition the annulus
$$
\bigl \{w \in \mathbb{D} : R-1 < d_{\mathbb{D}}(0, w) < R \bigr \}
$$
into $1 + \lceil 1/\gamma \rceil$ concentric annuli of hyperbolic widths $\le \gamma$.
\end{proof}

\part{Centered One Component Inner Functions}
\label{part:centered-one-component}

We say that an inner function $F$ is {\em singular} at a point $\zeta \in \partial \mathbb{D}$ if it does not admit any analytic extension to a neighbourhood of $\zeta$. 
Let $\Sigma \subset \partial \mathbb{D}$ be the set of singularities of $F$. It is clear from this definition that $\Sigma$ is a closed set.
While one usually thinks of inner functions as holomorphic self-maps of the unit disk, one may also view $F$ as a meromorphic function 
on $ \hat{\mathbb{C}} \setminus \Sigma$. 

In this work, we say that an inner function $F$ is a one component inner function if there is an annulus $\widetilde{A} = A(0; \rho, 1/\rho)$ such that $F: \sph \setminus \Sigma \to \sph$ is a covering map over $\widetilde{A}$. For the equivalence of this definition with the two definitions from the introduction, we refer the reader to \cite{inner-tdf}.

Throughout Part \ref{part:centered-one-component}, we assume that $F$ is a centered one component inner function of finite Lyapunov exponent that is not a rotation. We denote the class of all such inner functions by $\Lambda$.

In Section \ref{sec:fbp}, we define the Riemann surface lamination $\widehat{X}$ associated to $F$, as well as the geodesic and horocyclic flows on $\widehat{X}$.
In Section \ref{sec:almost invariant}, we discuss almost invariant functions on the unit disk and explain how one can derive orbit counting results from ergodicity and mixing of the geodesic flow.

In Section \ref{sec:mixing-geodesic}, we show that the horocyclic flow is ergodic and deduce that the geodesic flow is mixing.

\section{Background on Laminations}
\label{sec:fbp}

The {\em solenoid} associated to an inner function $F\in\Lambda$ is defined as the inverse limit
$$
\widehat{S^1} \, = \,  \lim_{\longleftarrow} \, ( F : S^1 \to S^1 )
\, = \,  \bigl \{ (u_i)_{i=-\infty}^0 : F(u_i) = u_{i+1} \bigr \}.
$$ 
In other words, a point on the solenoid is given by a point $u_0$ on the unit circle together with a consistent choice of pre-images $u_{-n} = F^{-n}(u_0)$.

Similarly, we can form the space of backwards orbits of $F$ on the unit disk
$$
\widehat{\mathbb{D}} \, = \,  \lim_{\longleftarrow} \, ( F : \mathbb{D} \to \mathbb{D} )\setminus \{ \mathbf{0} \}
\, = \,  \bigl \{ (z_i)_{i=-\infty}^0 : F(z_i) = z_{i+1} \bigr \} \setminus \{ \mathbf{0} \},
$$ 
where ${\bf 0} = \dots \leftarrow 0 \leftarrow 0 \leftarrow 0$ is the constant sequence. As we have removed the constant sequence $\bf 0$, each backward orbit tends to the unit circle, i.e.~$|z_{i}| \to 1$ as $i \to -\infty$.

For both $\widehat{S^1}$ and $\widehat{\mathbb{D}}$, we write $\pi_{-n}$ for the projection onto the $(-n)$-th coordinate, i.e.~the map
$
 (z_i )_{i=-\infty}^0 \mapsto z_{-n}.
 $

Let $\widehat{F}: \widehat{\mathbb{D}} \to \widehat{\mathbb{D}}$ be the map which applies $F$ to each coordinate.
Its inverse 
$
 (z_i )_{i=-\infty}^0 \mapsto (z_{i-1})_{i=-\infty}^0
$
is often called the {\em shift map}\/. The quotient 
$$
\widehat{X} = \widehat{\mathbb{D}} \setminus \widehat F
$$ 
is called the {\em Riemann surface lamination} associated to $F$. 

The term Riemann surface lamination refers to the fact that $\widehat{X}$ is locally homeomorphic to $\mathbb{D} \times \mathcal C$, where $\mathcal C$ is some topological space. By contrast, the solenoid $\widehat{S^1}$ is locally homeomorphic to $(-1, 1) \times \mathcal{C}$. 
When $F$ is a finite Blaschke product, the fiber $\mathcal C$ is a Cantor set, while if $F$ is an infinite-degree one component inner function, then $\mathcal C$ is homeomorphic to the shift space on infinitely many symbols $\{ 1, 2, \dots \}^{\mathbb{N}}$. In particular, the lamination $\widehat{X}$ is a Polish space, that is, a separable completely metrizable topological space. A particular complete metric compatible with the topology will be given in Section \ref{sec:metric-on-lamination}.

We now describe a particularly convenient collection of local charts or {\em flow boxes}  for $\widehat{X}$. Take a ball $\mathscr B = B(a, r)$ contained in the annulus $A \bigl (0; \frac{1+\rho}{2}, 1 \bigr )$ such that $F^{\circ n}(\mathscr B) \cap \mathscr B = \emptyset$ for any $n \ge 1$. 
Under this assumption, the sets $\{ F^{-n}(\mathscr B) \}_{n \ge 0}$ are disjoint.
Furthermore, by Koebe's distortion theorem, for any $n \ge 0$, the connected components of $F^{-n}(\mathscr B)$ are approximately round balls that are conformally mapped onto $\mathscr B$ by $F^{\circ n}$. Let 
$$
\widehat{\mathscr B}:=\pi_0^{-1}(\mathscr B) \subset \widehat{X},
$$ 
i.e.~$\widehat{\mathscr B}$ is the collection of all inverse orbits $\mathbf{z} = (z_i)_{i=-\infty}^0$ with $z_0 \in \mathscr B$.  For a finite Blaschke product, one needs finitely many such flow boxes to cover $\widehat{X}$ but for one component inner functions, which are not finite Blaschke products, one needs countably many. 

\subsection{Transverse measures}
\label{sec:transversals}

For a point $z \in \mathbb{D}$, the transversal $T(z)$ is defined as the collection of inverse orbits ${\bf w}$ with $w_0 = z$. If $w$ is a repeated pre-image of $z$, we write $T(w, z) \subset T(z)$ for the subset of inverse orbits which pass through $w$.
We define the Nevanlinna counting measure on $T(z)$ by specifying it on the ``cylinder'' sets $T(w, z) \subset T(z)$, where $w$ ranges over repeated pre-images of $z$:
$$
c_z(T(w,z)) = \log \frac{1}{|w|}.
$$
We also define the normalized counting measure by
$$
\overline{c}_z(T(w,z)) = \frac{\log \frac{1}{|w|}}{\log \frac{1}{|z|}}.
$$
If $z \in \mathbb{D}$ is not a repeated pre-image of an exceptional point, then $\overline{c}_z$ is a probability measure on $T(z)$.
By Frostman's theorem, this holds for all but a logarithmic capacity zero set of points in the unit disk.

\subsection{Linear structure}
\label{sec:1c-linear-structure}

We now show that each connected component or {\em leaf} of $\widehat{\mathbb{D}}$ associated to a one component inner function from the class $\Lambda$ is conformally equivalent to $(\mathbb{H}, \infty)$, while leaves of the solenoid $\widehat{S^1}$ are homeomorphic to the real line $\mathbb{R} \cong \partial \mathbb{H}$.

The marked point at infinity provides $\mathbb{H}$ with a sense of an upward direction: one can define the {\em upward-pointing vector field} $v_\uparrow(z) = y \cdot \frac{\partial}{\partial y}$ on $\mathbb{H}$. Indeed, $v_\uparrow$ is well-defined since it is invariant under 
$$
\aut (\mathbb{H}, \infty) = \bigl \{ z \mapsto Az + B: \, A > 0, \, B \in \mathbb{R} \bigr \}.
$$ 

As backward iteration is essentially linear near the unit circle, one may define an action of the half-plane $\mathbb{H}$ on $\widehat{X}$  by
\begin{equation}
\label{eq:H-action}
L({\bf z}, w)_j := \lim_{n \to \infty} F^{\circ n}(Z_{j-n}(w)),
\end{equation}
where
$$
Z_j(w) = \frac{z_j}{|z_j|} + \biggl ( z_j - \frac{z_j}{|z_j|} \biggr ) \frac{w}{i}.
$$
With this definition, $L({\bf z}, i) = {\bf z}$ while the leaf $\mathcal L$ of $\widehat{X}$ containing ${\bf z}$ is given by  $\{ L({\bf z}, w) : w \in \mathbb{H} \}$. 

By restricting $w$ to the imaginary axis, we obtain the {\em geodesic flow} on $\widehat{\mathbb{D}}$:
\begin{equation}
\label{eq:H-action2}
g_t({\bf z}) := L({\bf z}, e^t i),  \qquad t \in \mathbb{R}.
\end{equation}
By instead restricting $w$ to the line $\{\im w = 1\}$, we obtain the {\em horocyclic flow} on $\widehat{\mathbb{D}}$:
\begin{equation}
\label{eq:H-action3}
h_s({\bf z}) := L({\bf z}, i+s),  \qquad s \in \mathbb{R}.
\end{equation}
The two flows satisfy the relation
\begin{equation}
\label{eq:gh-commutation-relation}
g_{-t} h_s({\bf z}) = h_{e^t s} g_{-t}({\bf z}), \qquad s, t \in \mathbb{R}.
\end{equation}

The leaves of $\widehat{X}$ are hyperbolic Riemann surfaces covered by $(\mathbb{H}, \infty)$. In fact, most leaves are conformally equivalent to $(\mathbb{H}, \infty)$. The only exceptions are leaves associated to repelling periodic orbits on the unit circle. In this case, one needs to quotient $(\mathbb{H}, \infty)$ by multiplication by the multiplier of the repelling periodic orbit. See Section \ref{sec:dense-horocycle} for details.

It is easy to see that the geodesic and horocyclic flows descend to the Riemann surface lamination $\widehat{X}$.
In Section \ref{sec:mixing-geodesic}, we will see that unless $F(z) = z^d$ for some $d \ge 2$, the geodesic flow on $\widehat{X}$ is mixing, while the horocyclic flow on  $\widehat{X}$ is ergodic. In the exceptional case, the geodesic flow will be ergodic but not mixing.

\subsection{Natural measures}
\label{sec:natural-measures}

We endow the solenoid with the probability measure $\widehat{m}$ obtained by taking the natural extension of the Lebesgue measure on the unit circle with respect to the map $F: S^1 \to S^1$. The measure $\widehat m$ which is uniquely characterized by the property that its pushforward under any coordinate function $\pi_i: \widehat{S^1} \to S^1$, $i\in-\N_0$, is equal to $m$. Equivalently, $\widehat{m}$ is the unique $\widehat{F}$-invariant measure on $\widehat{S^1}$ whose  pushforward under $\pi_0$ is equal to $m$. As the Lebesgue measure $m$ on the unit circle is ergodic for $F:S^1\to S^1$, the measure $\widehat m$ is ergodic for $\widehat{F}: \widehat{S^1} \to \widehat{S^1}$.

We define a natural measure on the Riemann surface lamination $\widehat{X}$ by
$$
d\xi \, = \, \widehat m \times (dy/y) = c_z \, \times \, \frac{dxdy}{y^2},
$$
of total mass $\int_{S^1} \log |F'| \, dm$, where $ x + iy$ is an affine parameter on each leaf of $\widehat{X}$. Note that $dy/y$ is a well-defined 1-form on the Riemann surface lamination since it is invariant under $\aut (\mathbb{H}, \infty)$.
By construction, $d\xi$ is invariant under the geodesic and horocyclic flows on $\widehat{X}$.

For a measurable set $A$ contained in the unit disk, we write $\widehat{A}$ for the collection of inverse orbits ${\bf z}$ with $z_0 \in A$. By Koebe's distortion theorem, we have:

\begin{lemma}
\label{xi-measure-of-a-cylinder}
For a measurable set $A$ contained in the annulus $A \bigl (0; \frac{1+\rho}{2}, 1 \bigr )$,
$$\xi(\widehat{A}) \asymp \int_A \frac{dA(z)}{1-|z|}.$$
In fact, for any $\varepsilon > 0$, there exists an $ \frac{1+\rho}{2} < \rho' < 1$ so that
$$
(1 - \varepsilon) \cdot \frac{1}{2\pi} \int_A \frac{dA(z)}{1-|z|} \, \le \, 
\xi(\widehat{A})\, \le \, (1 + \varepsilon) \cdot \frac{1}{2\pi} \int_A \frac{dA(z)}{1-|z|}$$
for any measurable set $A \subset A(0; \rho', 1)$.
\end{lemma}

\subsection{Exponential coordinates and the suspension flow}
\label{sec:suspension-flow}

 In order to show that the geodesic flow $g_s: \widehat{X} \to\widehat{X}$ is ergodic, McMullen  \cite[Theorem 10.2]{mcmullen} relates it to a suspension flow over the solenoid. Let $\rho(z) = \log|F'(z)|$. The {\em suspension space}
$$
\widehat{S^1_\rho} \, = \, \widehat{S^1} \times \mathbb{R}_+ \, / \, \bigl ((z,t) \sim (F(z), e^{\rho(z)} \cdot t ) \bigr )
$$ 
carries a natural measure $\widehat{m}_\rho = \widehat{m} \times (dt/t)$ that is invariant under the {\em suspension flow} $\sigma_s : \widehat{S^1_\rho} \to \widehat{S^1_\rho}$ which takes
$(z,t) \to (z, e^s \cdot t)$.

 \begin{theorem}
\label{suspension-theorem}
The geodesic flow $(\widehat X, d\xi, g_s)$ on the Riemann surface lamination is equivalent to the suspension flow $(\widehat{S^1_\rho}, \widehat{m}_\rho, \sigma_s)$ on the suspension of the solenoid with respect to the roof function  $\rho = \log|F'|$.
\end{theorem}

\begin{proof}[Sketch of proof] 
The isomorphism between $\widehat{S^1} \times \mathbb{R}_+$ and $\widehat{\mathbb{D}}$ is given by the  {\em exponential map}
\begin{equation}
\label{eq:exponential-map-def}
E(\mathbf{u}, t) = \lim_{n \to \infty} F^{\circ n}(u_{i-n} + v_{i-n}),
\end{equation}
 where
$$
v_{i - n} = - \frac{t \cdot u_{i-n}}{|(F^{n-i})'(u_{i-n})|}.
$$
By Koebe's distortion theorem,
\begin{equation}
\label{eq:exponential-map}
E(\mathbf{u}, t)_{i} = u_{i} + v_i + o(|v_i|).
\end{equation}
In these exponential coordinates, the geodesic flow $g_s: \widehat{\mathbb{D}} \to  \widehat{\mathbb{D}}$ takes the form
\begin{equation}
\label{eq:intertwines-suspension-geodesic-flows}
g_s(E(\mathbf{u},t))= E(\mathbf{u}, e^{s} \cdot t).
\end{equation}
As a result, the exponential map descends to an isomorphism between $\widehat{S^1_\rho}$ and $\widehat{X}$ and intertwines the geodesic and suspension flows.
\end{proof}

 Since $m$ is ergodic under $F$ on the unit circle, $\widehat m$ is ergodic under $\widehat F$ on the solenoid and $\widehat{m}_\rho$ is ergodic under the suspension flow on $\widehat{S^1_\rho}$. The above theorem then implies that the geodesic flow on $\widehat{X}$ is ergodic.

\begin{remark}
The presentation of this section is inspired by \cite[Section 10]{mcmullen}. In Part \ref{part:general-inner-function}, we will give another perspective on the measure $\xi$ and the geodesic and horocyclic flows on $\widehat{X}$, in the context of general inner functions of finite Lyapunov exponent (which may not be one component).
\end{remark}

\section{Almost Invariant Functions}
\label{sec:almost invariant}

We say that a function $h: \mathbb{D} \to \mathbb{C}$ is {\em almost invariant} under $F$ if
$$
\limsup_{|F^{\circ n}(z)| \to 1}{|h(F^{\circ n}(z)) - h(z)|} = 0.
$$
In particular, for every backward orbit ${\bf z} = (z_i)_{i=-\infty}^0 \in \widehat{\mathbb{D}}$, $\lim_{i \to -\infty} h(z_i)$ exists and defines a  function on the Riemann surface lamination:
$$
\widehat{h}({\bf z}) = \lim_{i \to -\infty} h(z_i).
$$

\subsection{Consequences of ergodicity and mixing}
\label{sec:consequences-ergodicity-mixing}

In the following two theorems, we use ergodicity and mixing of the geodesic flow on $\widehat{X}$ to study almost invariant functions. The first theorem is a slight generalization from \cite[Theorem 10.6]{mcmullen}, which was originally stated for finite Blaschke products. For the convenience of the reader, we describe its proof in the setting of one component inner functions.

\begin{theorem}
\label{ergodic-theorem}
Let $F \in \Lambda$ be a one component inner function for which the geodesic flow on $\widehat{X}$ is ergodic.
Suppose $h: \mathbb{D} \to \mathbb{C}$ is a bounded almost invariant function that is uniformly continuous in the hyperbolic metric. Then for almost every $\zeta \in S^1$, we have
$$
\lim_{r \to 1} \frac{1}{|\log(1-r)|} \int_0^r h(s\zeta) \cdot \frac{ds}{1-s} = \fint_{\widehat X} \widehat{h} d\xi.
$$
In particular,
$$
\lim_{r \to 1} \frac{1}{2 \pi |\log(1-r)|} \int_{\mathbb{D}_r} h(z) \cdot \frac{dA(z)}{1-|z|} = \fint_{\widehat X} \widehat{h} d\xi.
$$
\end{theorem}

\begin{proof}
The ergodic theorem tells us that for almost every ${\bf{u}} \in \widehat{S^1}$, the backward time averages
$$
\lim_{T \to 0} \frac{1}{|\log T|} \int_T^1 \widehat{h}(E({\bf{u}}, t)) \cdot \frac{dt}{t} = \fint_{\widehat X} \widehat{h} d\xi.
$$
Write ${\bf{z}}(t) = E({\bf{u}},t)$. By almost-invariance, we have
$$
\widehat{h}(E({\bf{u}}, t))= h(z_0(t)) + o(1), \qquad \text{as }t \to 0^+,
$$
while
$$
h(z_0(t))= h((1-t)u_0) + o(1), \qquad \text{as }t \to 0^+,
$$
by  (\ref{eq:exponential-map}) and the uniform continuity of $h$ in the hyperbolic metric. Consequently,
$$
\lim_{T \to 0} \frac{1}{|\log T|} \int_T^1 \widehat{h}((1-t) u_0) \cdot \frac{dt}{t} = \fint_{\widehat X} \widehat{h} d\xi.
 $$
The proof is completed after making the change of variables $s = 1-t$ and relabeling $\zeta = u_0$ and $T = 1-r$.
\end{proof}

\begin{theorem}
\label{mixing-theorem}
Let $F \in \Lambda$ be a one component inner function for which the geodesic flow on $\widehat{X}$ is mixing.
 Suppose $h: \mathbb{D} \to \mathbb{C}$ is a bounded almost invariant function that is uniformly continuous in the hyperbolic metric. Then,
$$
\lim_{r \to 1}  \int_{|z|=r} h(z) dm = \frac{1}{\int_{S^1} \log|F'| dm} \int_{\widehat X} \widehat{h} d\xi.
$$
\end{theorem}

\begin{proof}
Consider a thin annulus
$$
A \, = \, A_{\hyp}(0; R_0, R_0 + \delta) \, = \,  \{w: R_0 < d_{\mathbb{D}}(0, w) \, < \,  R_0 + \delta \} \, \subset \, \mathbb{D}$$ of hyperbolic width $\delta$. Let $\widehat{A} \subset \widehat{X}$ be the collection of backwards orbits that pass through $A$. Since the geodesic flow is mixing, we have that
\begin{equation}
\label{eq:mixing-equation}
\lim_{t\to \infty}\frac{1}{\xi(\widehat{A})} \cdot \langle \chi_{\widehat{A}} \circ g_{t}, h \rangle =
\frac{1}{\int_{S^1} \log|F'| dm} \int_{\widehat{X}} \widehat{h} d\xi.
\end{equation}
In view of Lemma \ref{xi-measure-of-a-cylinder}, when $R_0 > 0$ is large, $$\xi(\widehat{A}) \, \approx \, \frac{1}{2\pi} \int_A \frac{dA(z)}{1-|z|} \, \approx \, \delta,
 \qquad \chi_{\widehat{A}} \circ g_{t} \approx \chi_{\widehat{A_t}},
 $$
where $A_t = A_{\hyp}(0;  R_0+t, R_0+t+\delta)$. Therefore, by the almost invariance of $h$, the left hand side of (\ref{eq:mixing-equation}) is approximately
$$
\frac{1}{2\pi \delta} \int_{A_t} h(z) \, \frac{dA(z)}{1-|z|} .
$$
When $\delta > 0$ is small, by the uniform continuity of $h$, this is approximately
$$\int_{\partial B_{\hyp}(0,R_0+t)} h(z) dm$$
as desired.
\end{proof}

\subsection{Orbit counting in presence of mixing}
\label{sec:1c-orbit-counting}

For a point $z \in \mathbb{D}$ sufficiently close to the unit circle and $0 < \delta < 1$, we construct an almost invariant function $h_{z, \delta}$ concentrated on a hyperbolic $O(\delta)$-neighbourhood of the inverse images of $z$:

\begin{enumerate}

\item By a {\em box} in the unit disk, we mean a set of the form
$$
\square = \{ w \in \mathbb{D} \, : \, \theta_1 < \arg w < \theta_2, \, r_1 < |w| < r_2 \}.
$$
For a point $z$ with $|z| > 1/2$ and $\delta > 0$ small, we write $\square = \square(z, \delta)$ for the box centered at $z$ of hyperbolic height $\delta$ and hyperbolic width $\delta$.

\item Recall that a one component inner function $F$ acts as a covering map over an $\widetilde{A} = A(0; \rho, 1/\rho)$. In particular, when $|z|$ is close to 1 and $\delta$ is small, the repeated pre-images $F^{-n}(\square)$ consist of disjoint squares of roughly the same hyperbolic size as the original, albeit distorted by a tiny amount. Define $h_{\rough}(w) = 1$ if $w \in F^{-n}(\square)$ for some $n \ge 0$ and $h_{\rough}(w) = 0$ otherwise.

\item We now smoothen the function from the previous step. To that end, consider a slightly smaller box $\square_2 = \square(z, \delta - \eta)$ with $\eta <\!\!< \delta$. Define $h_{z,\delta}$ to be a smooth function on $\square$ which is $1$ on $\square_2$, $0$ on $\partial \square$, and takes values between 0 and 1. Extend $h_{z,\delta}$ to $\bigcup_{n \ge 1} F^{-n}(\square)$ by backward invariance. Finally, extend $h_{z,\delta}$ by $0$ to the rest of the unit disk. Using the Schwarz lemma, it is not hard to see that $h_{z,\delta}$ is uniformly continuous in the hyperbolic metric.
\end{enumerate}

\begin{theorem}
\label{main-thm2a}
Let $F \in \Lambda$ be a one component inner function for which the geodesic flow on $\widehat{X}$ is mixing.
 Suppose $z \in \mathbb{D} \setminus \{ 0 \}$ lies outside a set of countable set. Then,
\begin{equation}
\label{eq:main-thm2a}
\mathcal N(z, R) \, \sim \, \frac{1}{2} \log \frac{1}{|z|} \cdot \frac{1}{\int_{\partial \mathbb{D}} \log |F'| dm} \cdot e^R,
\end{equation}
as $R \to \infty$.
\end{theorem}

\begin{proof}
We will show (\ref{eq:main-thm2a}) for any point $z \in \mathbb{D} \setminus \{ 0 \}$ which does not belong to a forward orbit of an exceptional point of $F$. From the results of Ahern and Clark discussed in Section \ref{sec:inner}, it is easy to see that this set is at most countable.

Below, we write $A \sim_\varepsilon B$ if
$$
1 - C\varepsilon \, \le \, A/B \, \le \, 1 + C\varepsilon,
$$
for some constant $C$ depending only on the inner function $F$ (and not on $z$ or $R$). More generally, we use the notation
$A \sim_{\varepsilon,\delta, R} B$ to denote that
$$
(1-o(1)) (1 - C\varepsilon) \, \le \, A/B \, \le \, (1+o(1))(1 + C\varepsilon)
$$
as $\delta \to 0^+$ and $R \to \infty$.

\medskip

{\em Step 1.} Suppose $z \in A(0; 1 - \varepsilon, 1)$ where $\varepsilon > 0$ is sufficiently small so the function $h_{z, \delta}$ is defined. In this step, we show that
\begin{equation}
\label{eq:z-close-to-the-unit-circle}
\mathcal N(z, R) \, \sim_{\varepsilon,R} \, \frac{1}{2} \log \frac{1}{|z|} \cdot \frac{1}{\int_{\partial \mathbb{D}} \log |F'| dm} \cdot e^R.
\end{equation}
To this end, we apply Theorem \ref{mixing-theorem} with $h = h_{z,\delta}$. 
In view of Lemma \ref{xi-measure-of-a-cylinder},
$$
 \frac{1}{\int_{S^1} \log|F'| dm} \int_{\widehat X} \widehat{h}_{z,\delta}\, d\xi \, \sim_{\varepsilon, \delta} \,  \frac{1}{\int_{S^1} \log|F'| dm} \cdot \frac{\delta^2}{2\pi} \cdot \log \frac{1}{|z|}.
 $$
 Since hyperbolic distance $\delta$ along the circle $\partial B_{\hyp}(0,R)$ corresponds to Euclidean distance of roughly $(2/e^R)\delta$,
$$
\int_{\partial B_{\hyp}(0,R)} h_{z,\delta}\, dm \, \sim_{\varepsilon, \delta, R} \, \frac{1}{2\pi} \cdot \frac{2\delta}{e^R} \cdot \mathcal N(z, R-\delta, R),
$$
where
$$
\mathcal N(z, R-\delta, R) = \# \bigl \{ w \in A_{\hyp} (0; R-\delta, R) : F^{\circ n}(w) = z \text{ for some }n \ge 0 \bigr \}.
$$
Comparing the two equations above, we see that
$$
\mathcal N(z, R-\delta, R) \sim_{\varepsilon,\delta,R}  \frac{\delta}{\int_{S^1} \log|F'| dm} \cdot \log \frac{1}{|z|} \cdot \frac{e^R}{2}.
$$
Integrating with respect to $R$ and taking $\delta \to 0$ shows (\ref{eq:z-close-to-the-unit-circle}).
 
 \medskip
 
 {\em Step 2.} Let $z \in \mathbb{D} \setminus \{ 0 \}$ be an arbitrary point in the punctured unit disk, which is not contained in the forward orbit of an exceptional point. In view of Lemma \ref{minimal-translation}, for any $\varepsilon > 0$, one can find an integer $m \ge 0$, so that any $m$-fold pre-image of $z$ is contained in $A (0; 1 - \varepsilon, 1 )$.
 
 By Lemma \ref{sum-of-heights},
  $$
 \sum_{F^{\circ m}(w) = z} \log \frac{1}{|w|} = \log \frac{1}{|z|}.
 $$ 
 We choose a finite set of $m$-fold pre-images $G_m$ so that
 $$
 \sum_{w \in G_m} \log \frac{1}{|w|} > \log \frac{1}{|z|} - \varepsilon.
 $$
 By Step 1, there exists a constant $C > 0$ (depending on $F$) so that
 $$
 \mathcal N(z, R) \, \ge \, \sum_{w \in G_m} \mathcal N(w, R) \, \ge \, (1 - C\varepsilon) \cdot \frac{1}{2} \log \frac{1}{|z|} \cdot \frac{1}{\int_{\partial \mathbb{D}} \log |F'| dm} \cdot e^R,
 $$
  for any $R > R_0(F, z)$ sufficiently large, depending on the inner function $F$ and the point $z \in \mathbb{D} \setminus \{ 0 \}$.

\medskip

{\em Step 3.}
  It remains to prove a matching upper bound. We use the same $m\ge 0$ as in the previous step. For any $0 \le k \le m$, let $T_k$ denote the set of repeated pre-images of $z$ of order $k$. Since
 $$
 \sum_{w \in T_k} \log \frac{1}{|w|} = \log \frac{1}{|z|}
 $$
 is finite by  Lemma \ref{sum-of-heights}, one can find a finite set $G_k \subset T_k$ so that
\begin{equation}
\label{eq:bad-k}
 \sum_{w \in T_k \setminus G_k} \log \frac{1}{|w|} < \varepsilon/m.
 \end{equation}
 Let $G = \bigcup_{k=0}^m G_k$ and $B =  \bigcup_{k=0}^m (T_k \setminus G_k)$. A somewhat crude estimate shows that
$$
\mathcal N(z, R) \le |G| + \sum_{w \in G_m} \mathcal N(w, R) + \sum_{w \in B} \mathcal N(w, R).
$$
By Step 1, 
$$
\sum_{w \in G_m} \mathcal N(w, R) \, \le \, (1 + C\varepsilon) \cdot \frac{1}{2} \log \frac{1}{|z|} \cdot \frac{1}{\int_{\partial \mathbb{D}} \log |F'| dm} \cdot e^R,
$$
while $\sum_{w \in B} \mathcal N(w, R)$ can be estimated using (\ref{eq:bad-k}) and Lemma \ref{orbit-counting-a-priori}.
\end{proof}

\subsection{Orbit counting in presence of ergodicity}

We now explain how to use the ergodicity of the geodesic flow to show orbit counting up to a Ces\`aro average:

\begin{theorem}
\label{main-thm-a}
Let $F \in \Lambda$ be a one component inner function for which the geodesic flow on $\widehat{X}$ is ergodic.
If $z \in \mathbb{D} \setminus \{ 0 \}$ lies outside a countable set, then
\begin{equation}
\label{eq:main-thm-a}
\lim_{R\to+\infty}
\frac{1}{R} \int_0^R \frac{\mathcal N(z, S)}{e^S} dS 
=\frac{1}{2} \log \frac{1}{|z|} \cdot \frac{1}{\int_{\partial \mathbb{D}} \log |F'| dm}.
\end{equation}
\end{theorem}

As the proof follows the same pattern as that of Theorem \ref{main-thm2a}, we only sketch the differences.

\begin{proof}[Sketch of proof]  

{\em Step 0.} The theorem boils down to showing
\begin{equation}
\label{eq:orbit-counting-goal}
\frac{1}{R} \sum_{\substack{F^{n}(w) = z, \, n \ge 0 \\ w \in B_{\hyp}(0,R)}} e^{-d_{\mathbb{D}}(0, w)} \, \to \,
\frac{1}{2} \log \frac{1}{|z|} \cdot \frac{1}{\int_{\partial \mathbb{D}} \log |F'| dm},
\end{equation}
as $R \to \infty$. Indeed, once we show (\ref{eq:orbit-counting-goal}), the theorem follows from the following computation:
\begin{align*}
\frac{1}{R} \int_0^R \frac{\mathcal N(z, S)}{e^{S}} dS & = \frac{1}{R}\sum_{\substack{F^{n}(w) = z, \, n \ge 0 \\ w \in B_{\hyp}(0,R)}}  \, \int_{d_{\mathbb{D}}(0,w)}^R e^{-S} dS \\
& = \frac{1}{R} \sum_{\substack{F^{n}(w) = z, \, n \ge 0 \\ w \in B_{\hyp}(0,R)}} \, (e^{-d_{\mathbb{D}(0,w)}}  - e^{-R}) \\
& =   \frac{1}{R} \sum_{\substack{F^{n}(w) = z, \, n \ge 0 \\ w \in B_{\hyp}(0,R)}} \, e^{-d_{\mathbb{D}(0,w)}}  + o(1),
\end{align*}
where in the last step we have used the a priori bound  (\ref{eq:orbit-counting-a-priori}) to estimate the number of terms.

\medskip

{\em Step 1.} Suppose $z \in A(0; 1 - \varepsilon, 1)$ where $\varepsilon > 0$ is sufficiently small so the function $h_{z, \delta}$ is defined.
In this step, we show that
\begin{equation}
\label{eq:step1-proof2}
\frac{1}{R} \sum_{\substack{F^{n}(w) = z, \, n \ge 0 \\ w \in B_{\hyp}(0,R)}} e^{-d_{\mathbb{D}}(0, w)} \, \sim_{\varepsilon,R} \, \frac{1}{2} \log \frac{1}{|z|} \cdot \frac{1}{\int_{\partial \mathbb{D}} \log |F'| dm}.
\end{equation}
Applying Theorem \ref{ergodic-theorem} to the almost invariant function $h = h_{z,\delta}$, we get
\begin{equation}
\label{eq:step4-ergodic}
\lim_{R \to \infty} \biggl \{ \frac{1}{2\pi R} \int_{B_{\hyp}(0, R)} h(x) \, \frac{dA(x)}{1-|x|} \biggr \} = \frac{1}{\int_{\partial \mathbb{D}} \log |F'| dm} \int_{\widehat{X}} \widehat{h} d\xi.
\end{equation}
The left hand side of (\ref{eq:step4-ergodic}) is approximately
\begin{align*}
& \sim_{\varepsilon, \delta, R} \frac{1}{2\pi R} \sum_{\substack{F^{n}(w) = z, \, n \ge 0 \\ w \in B_{\hyp}(0,R)}} \int_{\square(w, \delta)} h(x)  \, \frac{dA(x)}{1-|x|} \\
& \sim_{\varepsilon, \delta, R}  \frac{1}{\pi R} \sum_{\substack{F^{n}(w) = z, \, n \ge 0 \\ w \in B_{\hyp}(0,R)}} e^{-d_{\mathbb{D}}(0,w)} \cdot \int_{\square(w,\delta)} h(x) \, \frac{dA(x)}{(1-|x|^2)^2}.
\end{align*}
Meanwhile, by Lemma \ref{xi-measure-of-a-cylinder}, the right hand side of (\ref{eq:step4-ergodic}) is more or less
\begin{align*}
& \sim_{\varepsilon, \delta} \frac{1}{2\pi \int_{\partial \mathbb{D}} \log |F'| dm} \int_{\square(z,\delta)} h(x) \cdot \frac{dA(x)}{1-|x|} \\
& \sim_{\varepsilon, \delta}   \frac{1}{2\pi \int_{\partial \mathbb{D}} \log |F'| dm} \cdot \log \frac{1}{|z|} \cdot \int_{\square(z,\delta)} h(x) \, \frac{dA(x)}{(1-|x|^2)^2}.
\end{align*}
As $h$ is almost invariant, we have
$$
\int_{\square(w,\delta)} h(x) \, \frac{dA(x)}{(1-|x|^2)^2} \, \sim_{\varepsilon, \delta} \, \int_{\square(z,\delta)} h(x) \, \frac{dA(x)}{(1-|x|^2)^2},
$$
for any repeated pre-image $w$ of $z$.
Putting the above equations together and taking $\delta \to 0^+$, we get (\ref{eq:step1-proof2}).

\medskip

{\em Step 2.} Let $z \in \mathbb{D} \setminus \{ 0 \}$ be an arbitrary point in the punctured unit disk, which is not contained in the forward orbit of an exceptional point. Arguing as in Step 2 of Theorem \ref{main-thm2a}, one can show that for any $\varepsilon > 0$,
\begin{equation}
\label{eq:step2-proof2}
\frac{1}{R} \sum_{\substack{F^{n}(w) = z, \, n \ge 0 \\ w \in B_{\hyp}(0,R)}} e^{-d_{\mathbb{D}}(0, w)} \, \ge \,  (1 - C\varepsilon) \cdot \frac{1}{2} \cdot \log \frac{1}{|z|} \cdot \frac{1}{\int_{\partial \mathbb{D}} \log |F'| dm},
\end{equation}
provided that $R > R_0(F, z)$ is sufficiently large, which may depend on the inner function $F$ and the point $z \in \mathbb{D} \setminus \{ 0 \}$.

\medskip

{\em Step 3.} Arguing as in Step 3 of Theorem \ref{main-thm2a}, it is not difficult to find a matching upper bond
\begin{equation}
\label{eq:step3-proof2}
\frac{1}{R} \sum_{\substack{F^{n}(w) = z, \, n \ge 0 \\ w \in B_{\hyp}(0,R)}} e^{-d_{\mathbb{D}}(0, w)} \, \le \,  (1 + C\varepsilon) \cdot \frac{1}{2} \cdot \log \frac{1}{|z|} \cdot \frac{1}{\int_{\partial \mathbb{D}} \log |F'| dm},
\end{equation}
for $R > R_0(F, z)$ is sufficiently large. As $\varepsilon > 0$ was arbitrary in Steps 2 and 3, the proof is complete.
\end{proof}

\section{Mixing of the Geodesic Flow}
\label{sec:mixing-geodesic}

In this section, $F \in \Lambda$ will be a centered one component inner function of finite Lyapunov exponent, which is not $z \to z^d$ for some $d \ge 2$. 
We will show that the horocyclic flow on $\widehat{X}$ is ergodic and the geodesic flow  on $\widehat{X}$ is mixing.
The proof proceeds in four steps.

\begin{enumerate}
\item One first shows that the multipliers of the repelling periodic orbits are not contained in a discrete subgroup of $\mathbb{R}^+$. This step has been completed in \cite[Section 5]{inner-tdf}. This provides a large supply of homoclinic orbits.

\item We use an argument of Glutsyuk \cite{glutsyuk} to show that the horocyclic flow has a dense trajectory.

\item We use an argument of  Coud\`ene \cite{coudene} to promote the existence of a dense horocycle to the ergodicity of the horocyclic flow.

\item Finally, we use the ergodicity of the horocyclic flow to show the mixing of the geodesic flow. This can be done as in the case of a hyperbolic toral automorphism.
\end{enumerate}

 \subsection{A metric on the lamination}
 \label{sec:metric-on-lamination}
 
In order to discuss uniformly continuous functions on $\widehat{X}$, we endow $\widehat{X}$ with a metric that is compatible with the topology described in Section \ref{sec:fbp}.
 For ${\bf z}, {\bf w} \in \widehat{\mathbb{D}}$, we define
$$
 d_{\widehat{\mathbb{D}}} \bigl ({\bf z}, {\bf w}) := \min_{n \in \mathbb{Z}} \, \bigl \{ \max( 1 - |z_{-n}|, 1 - |w_{-n}| \bigr ) + d_{\mathbb{D}}(z_{-n}, w_{-n})   \bigr \}.
$$
To define a metric on the lamination, we try to align the indices as closely as possible:
$$
 d_{\widehat{X}}({\bf z}, {\bf w}) := \min_{m \in \mathbb{Z}}  d_{\widehat{\mathbb{D}}}({\bf z}, \widehat{F}^{\circ m}({\bf w})).
$$
As the above metric is complete and separable,  $\widehat{X}$ is a Polish space, but it is not locally compact unless $F$ is a finite Blaschke product. 

\begin{lemma}
Any leaf $\mathcal L$ is dense in $\widehat{X}$.
\end{lemma}

\begin{proof}
Suppose ${\bf z} \in \mathcal L$ and we want to show that ${\bf w} \in \widehat{X}$ lies in the closure of $\mathcal L$. For all $n\ge 0$ sufficiently large,
the points $z_{-n}$ and $w_{-n}$ lie in the annulus $A(0; \rho, 1)$. Connect $z_{-n}$ and
$w_{-n}$ by a curve $\gamma$ that lies in $A(0; \rho, 1)$. Following the inverse orbit ${\bf z}$ along the curve $\gamma$, we come to a point
${\bf z'} \in \mathcal L$ which agrees with ${\bf w}$ up to $z'_{-n} = w_{-n}$. From the definition of $d_{\widehat{X}}$, it is clear that as $n \to \infty$, these inverse orbits converge to ${\bf w}$.
\end{proof}

\subsection{Finding a dense horocycle}
\label{sec:dense-horocycle}

Pick a repelling fixed point $\xi$ on the unit circle. Let $r = F'(\xi)$ be its multiplier; it is real and positive. The leaf $\mathcal L_\xi$ which consists of all backwards orbits that tend to $\xi$ is conformally equivalent to $\mathbb{H}/(\cdot\, r)$. Let ${\bf z} \in \mathcal L_\xi$ be a point in this leaf and consider the horocycle $H({\bf z}) = \{ h_s({\bf z}) : s \in \mathbb{R} \}$ passing through ${\bf z}$. The horocycle is just a horizontal line in  $\mathcal L_\xi \cong \mathbb{H}/(\cdot\, r)$. Lifting to the upper half-plane $\mathbb{H}$, we get countably many horizontal lines. 

\begin{lemma}
\label{dense-horocycle}
The horocycle $H({\bf z})$ is dense in the leaf $\mathcal L_\xi$ and hence dense in the lamination $\widehat{X}$.
\end{lemma}

We may view $\im H({\bf z})$ as a number in $\mathbb{R}^+ / (\cdot\, r)$.
Glutsyuk's idea was to modify the backward orbit ${\mathbf z} \in L_\xi$ to obtain a new orbit ${\bf w} \in \mathcal L_\xi$ with $d_{\widehat{X}}({\bf z}, {\bf w})$ small, so that 
$\im H({\bf w})$ is close to any given number  in $\mathbb{R}^+ / (\cdot\, r)$.

By a $\xi$-{\em homoclinic orbit} ${\bf x} \in \widehat{S^1}$, we mean an inverse orbit 
$$
 \dots \to x_{-3} \to x_{-2} \to x_{-1} \to x_0, \qquad x_{-n} \in S^1,
$$
on the unit circle so that
$$
x_0 = \xi, \qquad \lim_{n \to \infty} {x_{-n}} = \xi.
$$
We can view the ``multiplier''
$$
m({\bf x}) = \lim_{n \to \infty} \frac{(F^{\circ n})'({x}_{-n})}{r^n}
$$
as an element of  $\mathbb{R}^+ / (\cdot\, r)$.

\begin{lemma}
The multipliers of $\xi$-homoclinic orbits are dense in $\mathbb{R}^+ / (\cdot\, r)$. 
\end{lemma}

\begin{proof}
As explained in \cite{inner-tdf}, if $F \in \Lambda$ is a centered one component inner function of finite Lyapunov exponent, which is not $z \to z^d$ for some $d \ge 2$, then the multipliers of repelling periodic orbits on the unit circle span a dense subgroup of $\mathbb{R}^+$.

For simplicity of exposition, assume that there is a single repelling periodic orbit $F^{\circ k}(\eta) = \eta$ on the unit circle such that $(F^{\circ k})'(\eta)$ and $r$ span a dense subgroup of $\mathbb{R}^+$.
As the inverse iterates of a point are dense on the unit circle \cite[Lemma 3.4]{inner-tdf}, for any $\varepsilon > 0$, one can find a $\xi$-homoclinic orbit ${\bf x}$ which passes within $\varepsilon$ of $\eta$\,:
 $$
 \dots \to x_{-3} \to x_{-2} \to x_{-1} \to x_0, \qquad |x_{-n} - \eta| < \varepsilon.
 $$
We can form a new $\xi$-homoclinic orbit ${\bf x}^{(p)}$ which starts with
$$
x_{-n} \to \dots \to x_{-3} \to x_{-2} \to x_{-1} \to x_0,
$$
then follows the periodic orbit $F^{\circ k}(\eta) = \eta$ for $p k$ steps, where $p \ge 1$ is a positive integer, and then follows the tail of ${\bf x}$:
$$
 \dots \to x_{-n -3} \to x_{-n-2} \to x_{-n-1} \to x_{-n} \to  \dots.
$$
Above, ``to follow an inverse orbit'' means to use the same branches of $F^{-1}$ defined on balls $B(\zeta, 1-\rho)$, centered on the unit circle.
By construction, for any given $p \ge 1$, we can make
$m({\bf x}^{(p)})$ as close to $(F^{\circ k})'(\eta)^p \cdot m({\bf x})$ as we want by requesting $\varepsilon > 0$ to be small. By the assumption on the multiplier of $\eta$, the numbers $(F^{\circ k})'(\eta)^p \cdot m({\bf x})$ are dense in  $\mathbb{R}^+ / (\cdot\, r)$.
\end{proof}

\begin{proof}[Proof of Lemma \ref{dense-horocycle}]
Let ${\bf z} \in L_\xi$ be a backward orbit in the unit disk. We can form a new backward orbit ${\bf w}$ by keeping
$$
z_{-n+1} \to \dots \to z_{-3} \to z_{-2} \to z_{-1} \to z_0
$$
and approximating
$$
\dots \to z_{-n-3} \to z_{-n-2} \to z_{-n-1} \to z_{-n}
$$
with a $\xi$-homoclinic orbit
$$
 \dots \to x_{-3} \to x_{-2} \to x_{-1} \to x_0.
$$
In other words, for $m \ge 0$, we replace $z_{-n-m}$ with a point close to $x_{-m}$. By choosing $n \ge 0$ sufficiently large and the $\xi$-homoclinic orbit appropriately, this construction produces inverse orbits ${\bf w} \in \mathcal L_\xi$ as close to ${\bf z} \in \mathcal L_\xi$ as we want with $\im H({\bf w})$ prescribed to arbitrarily high accuracy in $\mathbb{R}^+ / (\cdot\, r)$.
\end{proof}

\subsection{Ergodicity of the horocyclic flow}

\begin{lemma}
\label{the-horocyclic-flow-is-ergodic}

Suppose $F \in \Lambda$ is a centered one component inner function of finite Lyapunov exponent, other than $F(z) \ne z^d$ with $d \ge 2$. The horocyclic flow $h_s$ on the Riemann surface lamination $\widehat{X}$ is ergodic.
\end{lemma}

Following Coud\`ene, for $t > 0$, we define the operators
\begin{equation}
\label{eq:coudene-operators}
\mathcal M_t f({\bf z}) = \int_0^1 f(g_{-\log t}(h_s({\bf z}))) ds
\end{equation}
on the space of uniformly continuous functions $UC(\widehat{X})$. 
Let 
$$
S_t f({\bf z}) = \int_0^t f(h_{s}({\bf z})) ds
$$ denote the integral along the trajectory of the horocyclic flow up to time $t$. The motivation
for the operators (\ref{eq:coudene-operators}) is the relation
$$
\frac{S_t f ({\bf z})}{t} = \mathcal M_t f(g_{\log t}({\bf z})),
$$
which follows from \eqref{eq:gh-commutation-relation} and a change of variables.

\begin{lemma}\label{l120230825}
Suppose $F \in \Lambda$ is a centered one component inner function of finite Lyapunov exponent, other than $F(z) \ne z^d$ with $d \ge 2$.
If $f$ is a bounded uniformly continuous function on $\widehat{X}$, then the functions $\{M_t f\}_{t \ge 0}$, defined on $\widehat{X}$, form a uniformly equicontinuous family.
\end{lemma}

\begin{proof}[Sketch of proof]
The point is that if we do not change the point ${\bf z}$ much, we also do not change the horocycle of length $t$ from the point $g_{-\log t}({\bf z})$ much.
While the length of the horocycle is increasing (we are running it for time $t$), we are also starting it from the point $g_{-\log t}({\bf z})$. 
Koebe's distortion theorem implies that the horocycles of length $t$ started at points
 $g_{-\log t}({\bf w})$, with  $d_{\widehat X}({\bf z}, {\bf w}) < \varepsilon$, are within $O(\varepsilon)$ of one another.
\end{proof}

\begin{proof}[Proof of Lemma~\ref{the-horocyclic-flow-is-ergodic}]

In view of Lemma \ref{l120230825}, the Arzela-Ascoli theorem tells us that any sequence of functions $\mathcal M_{t_k}f$ with $t_k\to\infty$ contains a subsequence that converges uniformly on compact subsets of  ${\widehat X}$ to a function in $UC(\widehat{X})$.
Our goal is to show that for a positive function $f \in UC(\widehat{X})$, any accumulation point $\overline{f}$ of $\mathcal M_t f$ as $t \to \infty$, is a constant function $c = c(f)$, which would necessarily be $\fint_{\widehat{X}} f d\xi$. Once we have done this, 
the rest is easy: as the functions $\mathcal M_t f$ converge uniformly on compact subsets of ${\widehat X}$ to $c$ as $t\to\infty$,
 they also converge to $c$ in $L^2(\widehat{X}, d\xi)$. 
Here we are using that the metric space ${\widehat X}$ is Polish, which implies that the measure $\xi$ is inner regular on open sets and so there exists an increasing sequence of compact sets $K_n \subset \widehat{X}$ such that $\xi(K_n) \to \xi(\widehat{X})$.
 Consequently, $S_t(f)/t \to c$ in $L^2(\widehat{X}, d\xi)$ and the flow $h_s$ is ergodic.

Let $\{ t_k \}$ be a sequence of times tending to infinity for which
$\mathcal M_{t_k} f$
converges uniformly on compact subsets to an accumulation point $\overline{f} \in UC(\widehat{X})$. Using the invariance of the measure $\xi$ under geodesic flow, we see that
$$
\lim_{k\to\infty}\| (1/t_k) S_{t_k}(f) - \overline{f} \circ g_{\log t_k} \|_{L^2(\widehat{X}, d\xi)}=0.
$$
According to von Neumann's ergodic theorem, there is an $h_s$-invariant $L^2$ function $Pf$ on ${\widehat X}$ such that
$$
\lim_{t\to\infty}\|(1/t)S_{t}(f) - Pf \|_{L^2(\widehat{X}, d\xi)}=0.
$$
From these two observations and the $g_t$-invariance of $\xi$, we get:
$$
\| \overline{f} - Pf \circ g_{-\log t_k} \|_{L^2(\widehat{X}, d\xi)} \, = \,
\| \overline{f} \circ g_{\log t_k}  - Pf \|_{L^2(\widehat{X}, d\xi)} \, \to \, 0, \quad \text{as }k \to \infty.
$$
The commutativity property of the geodesic and horocyclic flows \eqref{eq:gh-commutation-relation} shows that $Pf \circ g_{-\log t_k}$ is invariant under the horocyclic flow $h_s$.
Therefore, $\overline{f}$ must also be invariant under $h_s$. As $\overline{f}$ is a continuous function with a dense $h_s$-orbit, it must be constant. The proof is complete.
\end{proof}

\subsection{Mixing of the geodesic flow}

We now deduce the mixing of the geodesic flow from the ergodicity of the horocyclic flow:

\begin{lemma}
If $F \in \Lambda$ is a centered one component inner function of finite Lyapunov exponent, other than $F(z) \ne z^d$ with $d \ge 2$, then the geodesic flow $g_{-t}$ on the Riemann surface lamination $\widehat{X}$ with respect to the measure $\xi$ is mixing.
\end{lemma}

\begin{proof}
For $t\in\mathbb R$, the Koopman operator $[g_{-t}]u=u\circ g_{-t}$ acts isometrically on $L^2(\widehat{X})$. For $r>0$, let $S_r(u)$ be the average of $u\circ h_s$ over $s \in [-r,r]$, i.e.
$$
S_r(u)(x)=\frac{1}{2r}\int_{-r}^ru(h_s(x))ds.
$$
This defines a bounded linear operator $S_r:L^2(\widehat{X})\to L^2(\widehat{X})$.
The commutation relation (\ref{eq:gh-commutation-relation}) tells us that
\begin{equation}\label{120230826}
S_r[g_{-t}]=[g_{-t}]S_{e^t r}
\end{equation}
as operators on $L^2(\widehat{X})$.

Let $u, v \in C^b({\widehat X})$ be two bounded continuous functions of zero mean with respect to the measure $\xi$. Since $\xi$ is invariant with respect to both the horocyclic flow $h_s$ and the geodesic flow $g_{-t}$, by using Fubini's Theorem, we get for every $r>0$ and $t\in\mathbb R$ that
\begin{equation}\label{220230826}
\langle S_r u,[g_{-t}]v \rangle
= \langle u, S_r[g_{-t}]v \rangle 
=  \langle u,[g_{-t}] S_{e^t r} v \rangle
= \langle [g_{t}]u, S_{e^t r} v \rangle.
\end{equation}
As $\int_{\widehat{X}} vd\xi=0$, it follows from the ergodicity of the horocyclic flow and von Neumann's Ergodic Theorem that $S_{e^t r} v \to 0$ in $L^2(\widehat{X})$ as $t\to+\infty$. Since the set $\{[g_{t}]u:t\in\mathbb R\}$ is bounded in $L^2(\widehat{X})$, \eqref{220230826} tells us that
$$
\lim_{t\to+\infty}\langle S_r u,[g_{-t}]v \rangle=0,
$$
for any $r > 0$. As 
$$
\lim_{r\to 0}\| u - S_r u\|_{L^2(\widehat{X})}=0,
$$
we also have
$$
\lim_{t\to+\infty}\langle u,[g_{-t}] v \rangle =0.
$$ 
The result now follows from the density of $C^b({\widehat X})$ in $L^2(\widehat{X})$.
\end{proof}

\part{Background in Geometry and Analysis}

In this part of the manuscript, we gather some facts from differential geometry and complex analysis that will allow us to study the dynamics of inner functions with finite Lyapunov exponent.

In Section \ref{sec:curves-hyperbolic-space}, we see use (hyperbolic) geodesic curvature to estimate how much a curve in the unit disk deviates from a radial ray $[0, \zeta)$. In Section \ref{sec:mobius-distortion}, we define the M\"obius distortion of a holomorphic self-map $F$ of the unit disk and use it to estimate the curvature of $F([0,\zeta))$.

In Section \ref{sec:notions-of-distortion}, we give another interpretation of the  M\"obius distortion in terms of how much  $F^{-1}$ expands the hyperbolic metric and define the linear distortion of $F$. Finally, in Section \ref{sec:distortion-along-radial-rays}, we give a bound on the total linear distortion of $F$ along $[0,\zeta)$ in terms of the angular derivative $|F'(\zeta)|$, from which we conclude that if $F$ is an inner function with finite Lyapunov exponent then the total linear distortion of $F$ on the unit disk is finite.

\label{sec:curves-hyperbolic-space}

\section{Curves in Hyperbolic Space}

We first recall the definition and basic properties of geodesic curvature in the Euclidean setting. Suppose $\gamma: [a,b] \to \mathbb{R}^2$ is a $C^2$ curve, parameterized with respect to arclength.
Its curvature $$\kappa_{\Euc}(\gamma;t) = \| \gamma''(t) \|$$ measures the rate of change of the tangent vector of $\gamma$. The signed curvature $\kappa_{\signed,\,\Euc}(\gamma;t) = \pm\, \kappa_{\Euc}(\gamma;t)$ also takes into account if $\gamma$ is turning left or right. It is well known that a curve is uniquely determined (up to an isometry) by its signed curvature, e.g.~see \cite[Theorem 2.1]{pressley}.

\begin{example} A circle of radius $R$ has constant curvature $1/R$. The signed curvature is either $-1/R$ or $1/R$ depending on the orientation of $\gamma$.
\end{example}

We now turn our attention to the hyperbolic setting. Let $\gamma: [a, b] \to \mathbb{D}$ be a $C^2$ curve, parametrized with respect to hyperbolic arclength. The hyperbolic geodesic curvature $\kappa_{\hyp}(\gamma;t)$ measures how much $\gamma$ deviates from a hyperbolic geodesic at $\gamma(t)$.

We now describe a convenient way to compute  $\kappa_{\hyp}(\gamma;t)$.
Suppose first $\gamma$ passes through the origin, e.g.~$\gamma(t_0) = 0$ for some $t_0 \in [a,b]$.
As the hyperbolic metric osculates the Euclidean metric to order 2 at the origin, but is twice as large there, the hyperbolic geodesic curvature
of $\gamma$ is half the Euclidean geodesic curvature of $\gamma$.
One may compute the hyperbolic geodesic curvature at other points by means of $\aut (\mathbb{H})$ invariance.

\begin{example}
(i)  Hyperbolic geodesics have zero geodesic curvature.

(ii) To compute the curvature of a horocycle, we may assume that the horocycle passes through the origin and compute its curvature there.
Since a horocycle which passes through the origin is a circle of Euclidean radius $1/2$, its Euclidean geodesic curvature at the origin is $2$. Consequently, every horocycle has constant hyperbolic geodesic curvature $1$.

(iii) Curves of constant hyperbolic geodesic curvature $\kappa \in (0,1)$ are circular arcs which cut the unit circle at two points at an angle $\theta \in (0, \pi/2)$
with $\kappa = \cos \theta$.
\end{example}
 
The following two lemmas are well-known:

\begin{lemma}
If $\gamma: [a, b] \to \mathbb{D}$ is a $C^2$ curve with hyperbolic geodesic curvature $\kappa_{\hyp}(\gamma;t) \le 1$, then $\gamma$ is a simple curve.
\end{lemma}

\begin{lemma}
\label{quasigeodesic-property-of-hyperbolic-space}
If $\gamma: [a, b] \to \mathbb{D}$ is a $C^2$ curve with hyperbolic geodesic curvature $\kappa_{\hyp}(\gamma;t) \le c < 1$, then $\gamma$ lies within a bounded hyperbolic distance of some geodesic.
\end{lemma}

We also record the following comparison theorem:

\begin{theorem}
\label{comparison-theorem}
Suppose $\gamma: [a, \infty) \to \mathbb{D}$ is a $C^2$ curve with hyperbolic geodesic curvature $\kappa_{\hyp}(\gamma;t) \le \kappa \le 1$. Let $\gamma_1, \gamma_2: [a, \infty) \to \mathbb{D}$ be curves with constant signed geodesic curvatures $\kappa$ and $-\kappa$ respectively that have the same tangent vector at $t = a$, i.e.~
$$
\gamma_1(a) = \gamma_2(a) = \gamma(a), \qquad \gamma'_1(a) = \gamma'_2(a) = \gamma'(a).
$$
Then, $\gamma$ lies between $\gamma_1$ and $\gamma_2$.
\end{theorem}

\subsection{Inclination from the Vertical Line}

We now switch to the upper half-plane model of hyperbolic geometry. In this section, we assume that $\gamma: [a, \infty) \to \mathbb{H}$ is a $C^2$ curve of curvature $\kappa \le 0.2$, parametrized with respect to arclength. 
For any $a \le t < \infty$, we can look at the tangent vector $\gamma'(t)$ to $\gamma$ at the point $\gamma(t)$.
We define $\alpha(t) \in [0, \pi]$ to be the angle that $\gamma'(t)$ makes
with the  downward pointing vector field $v_{\downarrow} = - y \cdot \frac{\partial}{\partial y}$.

\begin{figure}[h]
\centering
\includegraphics[scale=0.6]{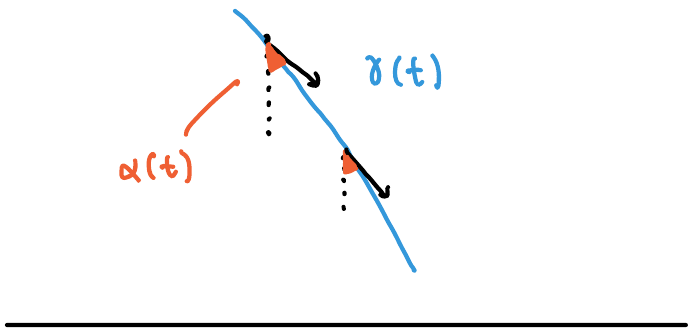}
 \caption{Inclination from the vertical line}
 \label{fig:curvature-lemma1}
\end{figure}

We first describe the behaviour of $\alpha(t)$ when $\gamma$ is a hyperbolic geodesic in the upper half-plane.
 Inspection shows that the  derivative $\alpha'(t) \le 0$, where equality holds if and only if $\gamma$ is a vertical line, pointing straight up or straight down. If $\gamma$ is not a vertical line, then $\alpha(t)$ satisfies the differential equation
$$
\alpha'(t) = -G(\alpha(t)),
$$
for some non-negative differentiable function $G: [0, \pi] \to \mathbb{R}$, which vanishes only at the endpoints. 
(The function $G$ does not depend on the geodesic $\gamma$ since any two non-vertical geodesics in the upper half-plane are related by a mapping of the form $z \to Az + B$ with $A > 0$ and $B \in \mathbb{R}$.) For future reference, we note that $G'(0) > 0$.

\begin{lemma}
\label{inclination-bound-geodesic-case}
Suppose $\gamma: [a, b] \to \mathbb{H}$ is a piece of a hyperbolic geodesic. If $\alpha(a) \le 2\pi/3$, then
\begin{equation}
\label{eq:inclination-bound-geodesic-case}
\int_a^b \alpha(t) \lesssim \alpha(a),
\end{equation}
where the implicit constant is independent of $b$.
\end{lemma}

\begin{proof}
From the discussion above, it follows that $\alpha(t)$ satisfies the differential inequality
$$
\alpha'(t) \le -c_1 \alpha(t), \qquad t \in [a,\infty),
$$
for some $c_1 > 0$. In view of Gr\"onwall's inequality, $\alpha(t)$ decreases exponentially quickly, which clearly implies (\ref{eq:inclination-bound-geodesic-case}).
\end{proof}

We now turn to investigating $\alpha(t)$ for general curves $\gamma$ with small geodesic curvature. We begin with the following preliminary observation:

\begin{lemma}\label{l120240109}
If $\gamma: [a, \infty) \to \mathbb{H}$ is a $C^2$ curve parametrized with respect to hyperbolic arclength with curvature $\kappa \le 0.2$.
If $\alpha(a) < 2\pi/3$, then 
$$
\alpha(t) \le 2\pi/3
$$
for all $t \in [a,\infty)$.
\end{lemma}

\begin{proof}[Sketch of proof]
From the discussion above, a straight line in the upper half-plane with $\alpha(t) = 2\pi/3$ has constant curvature $\kappa = \sqrt{3/2} > 0.2$. By Theorem \ref{comparison-theorem}, if $\alpha(t) = 2\pi/3$ then $\alpha'(t) \le 0$. Consequently, $\alpha(t)$ cannot rise above $2\pi/3$.
\end{proof}

\begin{lemma}
If $\gamma \subset \mathbb{H}$ is a $C^2$ curve parametrized with respect to hyperbolic arclength, with curvature $\le 0.2$, then
$$
\alpha'(t) \le - G(\alpha(t)) + 4 \, \kappa_{\hyp}(\gamma;t).
$$
\end{lemma}

\begin{proof}[Sketch of proof]
We have seen that at the origin, the hyperbolic metric is twice as large as the Euclidean metric. As a result, the parametrization with respect to the hyperbolic arclength is twice as fast as with Euclidean arclength. In addition, the Euclidean geodesic curvature is twice as large as the hyperbolic geodesic curvature. Consequently, the instrinsic change in the direction of the tangent vector $\gamma'(t)$ is four times the signed hyperbolic geodesic curvature. 

However, in hyperbolic geometry, we must also account for the fact that geodesics naturally change direction with respect to the vertical, which is described by the first term in the equation above.
\end{proof}

To conclude this section, we extend Lemma \ref{inclination-bound-geodesic-case} to the case of small geodesic curvature:

 \begin{lemma}
\label{inclination-bound}
Suppose $\gamma: [a, b] \to \mathbb{H}$ has geodesic curvature at most $0.2$. If $\alpha(a) \le 2\pi/3$, then
$$
\int_a^b \alpha(t) \lesssim \alpha(a) + \int_a^b k_{\hyp}(\gamma;t).
$$
\end{lemma}

\begin{proof}
From the lemma above, it follows that
\begin{equation}
\label{220240109}
\alpha'(t) \le -c_1 \alpha(t) + 4 \, \kappa_{\hyp}(\gamma;t),
\end{equation}
for some $c_1 > 0$. Gr\"onwall's inequality shows that
$$
\alpha(t) \le e^{-c_1 t} \biggl ( \alpha(a) + 4\int_a^t e^{c_1 s} \cdot \kappa_{\hyp}(\gamma;s) ds \biggr ),
$$
for all $t \in [a,b]$. Integrating over $t$ proves the result.
\end{proof}

\section{M\"obius Distortion}
\label{sec:mobius-distortion}

Let $\lambda_{\mathbb{D}} = \frac{2}{1-|z|^2}$ be the hyperbolic metric on the unit disk.
A holomorphic self-map $F$ of the unit disk naturally defines the pullback metric
$$
\lambda_F = F^* \lambda_{\mathbb{D}} = \frac{2|F'(z)|}{1-|F(z)|^2}.
$$
With the above definition, if $\gamma \subset \mathbb{D}$ is a rectifiable curve, then the hyperbolic length of $F(\gamma)$ is $\int_\gamma \lambda_F$.

By the Schwarz lemma, $\mu_F(z) := 1 - (\lambda_F/\lambda_{\mathbb{D}})(z)$ is zero if and only if $F$ is a M\"obius transformation.
In general, for any $a\in\mathbb{D}$, the {\em M\"obius distortion} $\mu_F(a)$ measures how much $F$ deviates from being a M\"obius transformation near $a$. A normal families argument shows that when $\mu_F(a)$ is small, $F$ is close to a M\"obius transformation $m \in \aut(\mathbb{D})$ near $a$. The following lemma provides a more quantitative estimate:

\begin{lemma}
\label{small-mobius-distortion}
Let $F$ be a holomorphic self-map of the unit disk.  For any $R, \varepsilon > 0$, there exists a $\delta > 0$, so that
if $\mu_F(a) < \delta$, $a \in \mathbb{D}$, then on $B_{\hyp}(a, R)$,
$F$ is univalent  and $d_{\mathbb{D}} \bigl (F(z), m(z) \bigr ) < \varepsilon$ for some M\"obius transformation $m \in \aut(\mathbb{D})$ which takes $a$ to $F(a)$. Furthermore, for a fixed $R > 0$, $\delta$ can be taken to be comparable to $\varepsilon$.
\end{lemma}

The argument below is taken from \cite[Proposition 10.9]{mcmullen-rtree}:

\begin{proof}
By M\"obius invariance, we may assume that $a = F(a) = 0$ and $0 \le F'(0) \le 1$. For convenience, we abbreviate $\mu = \mu_F(0) = 1 - F'(0)$.
Applying the Schwarz lemma to $F(z)/z$ shows that the hyperbolic distance
$$
d_{\mathbb{D}}(F(z)/z, F'(0)) = O(1), \qquad \text{for }z \in B_{\hyp}(0, R+1).
$$
 Taking note of the location of $F'(0) \in \mathbb{D}$, this implies that $|F(z) - z| = O(\mu)$ on $B_{\hyp}(0, R+1)$. Cauchy's integral formula then tells us that
 $$
 |F'(z) - 1| = O(\mu), \qquad  |F''(z)| = O(\mu), \qquad \text{for }z \in B_{\hyp}(0,R).
 $$
 From here, the lemma follows from the definition of the M\"obius distortion and some arithmetic.
\end{proof}

\begin{lemma}
\label{geodesic-curvature-and-mobius-distortion}
Suppose $\gamma$ is a hyperbolic geodesic in the unit disk passing through $z \in \mathbb{D}$.
Let $F(\gamma)$ be the image of $\gamma$ under a holomorphic self-map $F$ of the unit disk. Then the geodesic curvature of $F(\gamma)$ at $F(z)$ is bounded by
$$
\min \bigl (1, \kappa_{F(\gamma)}(F(z)) \bigr ) \lesssim \mu(z).
$$
\end{lemma}

\begin{proof}
By M\"obius invariance, one can consider the case when $\gamma = [-1,1]$, $z = 0$, $F(0) = 0$ and $F'(0) > 0$.
Arguing as in the proof of Lemma \ref{small-mobius-distortion}, we get $|F''(0)| = O(\mu)$ where $\mu = 1 - F'(0)$. Therefore, $F(\gamma)$ lies in a wedge
$$
\bigl \{ x + iy: |y| < C\mu x^2 \bigr  \}
$$
near $z=0$, which gives the desired curvature bound.
\end{proof}

The same argument shows:

\begin{lemma}[Stability of $\mu$ under perturbations]
\label{mu-derivative}
There exists a constant $K > 0$ so that any holomorphic self-map $F$ of the unit disk,
$$
|\nabla_{\hyp} \mu(a)| \le K \mu(a), \qquad a \in \mathbb{D}.
$$
In particular, for any two points $a, b \in \mathbb{D}$, we have 
$$
e^{-K d_{\mathbb{D}}(a, b)}  \mu(a) \, \le \, \mu(b) \, \le \, e^{K d_{\mathbb{D}}(a, b)}  \mu(a).
$$
\end{lemma}

 For a finer estimate, we refer the reader to \cite[Corollary 5.7]{beardon-minda}.

\paragraph*{Hyperbolic expansion factor.}

Suppose $F$ is a holomorphic self-map of the unit disk. By the Schwarz lemma, the {\em hyperbolic expansion factor} $E(a) := \| F'(a) \|_{\hyp}^{-1} \ge 1$. The hyperbolic expansion factor could be infinite if $a$ is a critical point of $F$. 
The hyperbolic expansion factor is related to the M\"obius distortion via
$$
E(a) = \frac{1}{1-\mu(a)}.
$$
As a result, the two quantities are essentially interchangeable.

\section{Linear Distortion}
\label{sec:notions-of-distortion}

Recall that the downward pointing vector field $v_{\downarrow} = - y \cdot \frac{\partial}{\partial y}$ assigns each point in $\mathbb{H}$ a vector of hyperbolic length 1 which points toward the real axis. By the Schwarz lemma, the quotient
$$
p(z) = \frac{F_* v_{\downarrow}(z)}{v_{\downarrow}(F(z))} \, \in \, \mathbb{D}.
$$
We consider the following quantities:
\begin{itemize}
\item M\"obius distortion: $\mu \, = \, 1 - |p|.$
\item Linear distortion: $\delta \, = \, |1 - p|.$
\item Vertical inefficiency: $\eta \, = \, \re(1 - p).$
\item Vertical inclination: $\alpha \, = \, |\arg p| \, \in \, [0, \pi).$
\end{itemize}

\begin{figure}[h]
\centering
\includegraphics[scale=0.6]{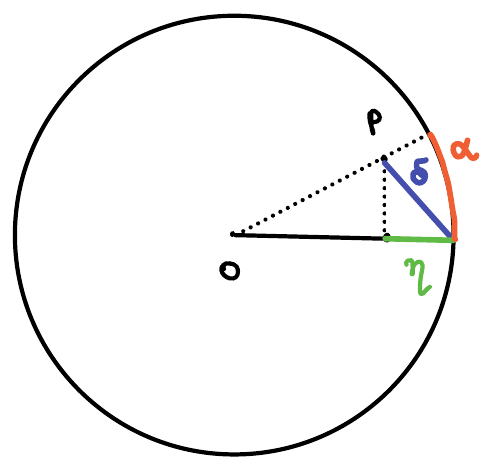}
 \caption{Notions of distortion}
 \label{fig:aed}
\end{figure}

In practice, estimating $\delta(a)$ directly is rather difficult. From the picture above, it is clear that $\alpha(a) + \eta(a) \ge \delta(a)$, which allows us to estimate linear distortion by estimating the vertical inefficiency and vertical inclination separately.

For a holomorphic self-map of the upper half-plane $F$, the linear distortion $\delta_F(a)$ measures how much $F$ deviates from the unique linear map 
$L_{a \to F(a)} \in \aut (\mathbb{H}, \infty)$ which takes $a$ to $F(a)$.
 Evidently, the linear distortion is zero if and only if $F = L_{a \to F(a)}$. Similarly to Lemma \ref{small-mobius-distortion}, we have:

\begin{lemma}
\label{small-linear-distortion}
Let $F$ be a holomorphic self-map of the upper half-plane. For any $R, \varepsilon > 0$, there exists a $\delta > 0$, so that
if $\delta_F(a) < \delta$, $a \in \mathbb{H}$, then on $B^{\mathbb{H}}_{\hyp}(a, R)$, $F$ is univalent and $d_{\mathbb{H}} \bigl (F(w), L_{a \to F(a)}(w) \bigr ) < \varepsilon$.
Furthermore, for a fixed $R > 0$, $\delta$ can be taken to be comparable to $\varepsilon$.
\end{lemma}

\begin{proof}[Proof]
By $\aut(\mathbb{H}, \infty)$ invariance, we may assume that $a = F(a) = i$.
In view of Lemma \ref{small-mobius-distortion}, $F$ is injective on the ball $B_{\hyp}^{\mathbb{H}}(a, R)$, where it resembles an elliptic M\"obius transformation in $\aut(\mathbb{H})$ which fixes $i \in \mathbb{H}$. We need to show that $F$ is close to the identity mapping.

Let $m_{\mathbb{D} \to \mathbb{H}}$ be a M\"obius transformation which maps $\mathbb{D}$ to $\mathbb{H}$ and takes the tangent vector $(\partial/\partial x)(0)$ to $v_\downarrow(i)$. As $\delta_F(i) < \delta$, the composition $$G = m_{\mathbb{H} \to \mathbb{D}} \circ F \circ m_{\mathbb{D} \to \mathbb{H}}$$ defines a holomorphic self-map of the unit disk 
with $G(0) =0$ and $|G'(0) - 1| = O(\delta_F(i))$. Following the proof of Lemma \ref{small-mobius-distortion}, we see that
$|G(z) - z| = O(\delta_F(i))$ for $z \in B_{\hyp}(0, R)$, which
in the upper half-plane translates to
\begin{equation}
\label{eq:small-displacement}
|F(w) - w| = O(\delta_F(i)), \qquad \text{for }w \in B^{\mathbb{H}}_{\hyp}(i, R),
\end{equation}
as desired.
\end{proof}

\begin{lemma}
\label{small-linear-distortion2}
In the lemma above, we may choose $\delta \asymp \varepsilon$ so that
\begin{equation}
\label{eq:small-linear-distortion2}
 (1 - \varepsilon) \cdot \frac{1}{\im w} \, < \,  \frac{|F'(w)|^2}{\im F(w)} \, < \, (1+ \varepsilon) \cdot \frac{1}{\im w},
\end{equation}
for any $w \in B^{\mathbb{H}}_{\hyp}(a, R)$.
\end{lemma}

\begin{proof}
We continue to use the normalization $a = F(a) = i$. In view of  (\ref{eq:small-displacement}) and Cauchy's integral formula, there exists a constant $C > 0$ depending only on $R$ so that 
$$
1 - C\delta_F(i)  \, < \,  |F'(w)| \, < \, 1+C\delta_F(i),
$$
for any $w \in B_{\hyp}^{\mathbb{H}}(a, R-1)$. Together with (\ref{eq:small-displacement}), this shows
$$
 (1 - C' \delta_F(i)) \cdot \frac{1}{\im w} \, < \,  \frac{|F'(w)|^2}{\im F(w)} \, < \, (1+ C' \delta_F(i)) \cdot \frac{1}{\im w},
$$
for some constant $C' > 0$ which also only depends on $R$. To obtain the corollary as stated, we work with $R+1$ in place of $R$ and divide $\delta$ by a constant if necessary so that $C' \delta < \varepsilon$.
\end{proof}

\begin{lemma}[Stability of $\delta$ under perturbations]
\label{delta-derivative}
There exists a constant $K > 0$ such that for any holomorphic self-map $F$ of the upper half-plane,
$$
|\nabla_{\hyp} \delta(a)| \le K \delta(a), \qquad a \in \mathbb{H}.
$$
In particular, for any two points $a, b \in \mathbb{H}$, we have 
$$
e^{-K d_{\mathbb{D}}(a, b)}  \delta(a) \, \le \, \delta(b) \, \le \, e^{K d_{\mathbb{D}}(a, b)}  \delta(a).
$$
\end{lemma}

\begin{proof}
 Since we have already estimated $|\nabla_{\hyp} \mu(a)|$ in Lemma \ref{mu-derivative}, it remains to control the gradient of the angular inclination 
 $|\nabla_{\hyp} \alpha(a)|$. 
  By $\aut(\mathbb{H}, \infty)$ invariance, we may assume that $a = F(a) = i$. We may also assume that $\delta(i) < 1/2$, otherwise the lemma is trivial, in which case, $|F'(i)| \asymp 1$ and the gradient 
$|\nabla \alpha(i)| = |\nabla \arg F'(i)|$ is controlled by the second derivative $|F''(i)|$.

As in the proof of Lemma \ref{small-linear-distortion}, we have
$
|F(w) - w| = O(\delta(i))$ for $w \in B^{\mathbb{H}}_{\hyp}(i, 1).
$
An application of Cauchy's integral formula gives the desired estimate $|F''(i)| = O(\delta(i))$.
\end{proof}

\paragraph*{Working in the unit disk.} For a centered holomorphic self-map $F$ of the unit disk, one can define the notions of $\delta, \eta, \alpha$ using the {\em radial vector field} $v_{\rad}(z) = \frac{2}{1- r^2} \cdot \frac{\partial}{\partial r}$, which assigns each point in $\mathbb{D} \setminus \{ 0 \}$ an outward pointing vector of hyperbolic length 1. Note that $\delta, \eta, \alpha$ are only defined when $a, F(a) \ne 0$ and as a result are somewhat awkward to work with. Nevertheless, near the unit circle, $\delta, \eta, \alpha$ resemble their counterparts in the upper half-plane.

 Assuming that $a, F(a) \ne 0$, the radial distortion $\delta_F(a)$ measures how much $F$ deviates from  $m_{a \to F(a)}$ near $a$, the ``straight'' M\"obius transformation which takes
$$
a \to F(a), \qquad \frac{a}{|a|} \to \frac{F(a)}{|F(a)|}, \qquad  -\frac{a}{|a|} \to -\frac{F(a)}{|F(a)|}.
$$
To ensure that $F$ is close to the linear map $\ell_{a \to F(a)}$ which takes $a \to F(a)$ and $\frac{a}{|a|} \to \frac{F(a)}{|F(a)|}$ on $B_{\hyp}(a,R)$, we will often ask that $1-|F(a)| < \delta/e^R$ in addition to $\delta_F(a) < \delta$.

\section{Distortion Along Radial Rays}
\label{sec:distortion-along-radial-rays}

Suppose $F$ is a holomorphic self-map of the unit disk. Recall that $F$ has an {\em angular derivative} at $\zeta \in \partial \mathbb{D}$ in the sense of Carath\'eodory if $F(\zeta) := \lim_{r \to 1} F(r \zeta)$ belongs to the unit circle and $F'(\zeta) := \lim_{r \to 1} F'(r \zeta)$ is finite.
The following theorem says that the logarithm of the angular derivative at $\zeta$ controls the total linear distortion along the radial geodesic $[0, \zeta)$:

\begin{theorem}
\label{total-radial-distortion}
Suppose $F$ is a holomorphic self-map of the disk with $F(0) = 0$. If $F$ has an angular derivative at $\zeta \in \partial \mathbb{D}$, then
\begin{equation}
\label{eq:total-radial-distortion}
\int_0^\zeta \delta \, d\rho \lesssim \log |F'(\zeta)|.
\end{equation}
In particular, if $F$ is an inner function with finite Lyapunov exponent,
\begin{equation}
\label{eq:total-radial-distortion2}
\int_{\mathbb{D}}\delta(z) \cdot \frac{dA(z)}{1-|z|} \lesssim \int_{\partial \mathbb{D}} \log |F'(re^{i\theta})| dm.
\end{equation}
\end{theorem}

In view of the inequality $\alpha(z) + \eta(z) \ge \delta(z)$, we may split the proof Theorem \ref{total-radial-distortion} into two lemmas:

\begin{lemma}
\label{radial-inefficiency}
Suppose $F$ is a holomorphic self-map of the disk with $F(0) = 0$. If $F$ has an angular derivative at $\zeta \in \partial \mathbb{D}$, then
\begin{equation}
\label{eq:radial-inefficiency}
\int_0^\zeta \eta \, d\rho \le \log |F'(\zeta)|.
\end{equation}
\end{lemma}

\begin{lemma}
\label{radial-inclination}
Suppose $F$ is a holomorphic self-map of the disk with $F(0) = 0$. If $F$ has an angular derivative at $\zeta \in \partial \mathbb{D}$, then
\begin{equation}
\label{eq:radial-inclination}
\int_0^\zeta \alpha \, d\rho \lesssim \log |F'(\zeta)|.
\end{equation}
\end{lemma}

\subsection{Bounding the radial inefficiency}

We first estimate the radial inefficiency:

\begin{proof}[Proof of Lemma \ref{radial-inefficiency}]
Let $\zeta$ be a point on the unit circle where $F$ has an angular derivative. Join the points $0$ and $\zeta$ by a hyperbolic geodesic 
 $\gamma = [0, \zeta)$.  The image $F(\gamma)$ is a curve which connects $0$ to $F(\zeta) \in \partial \mathbb{D}$. From the definition of the radial inefficiency,
 it is clear that
$$
\int_{0}^\zeta \eta \, d\rho \, \le \,
 \lim_{r \to 1} \bigl \{ d_{\mathbb{D}}(0,r\zeta) -  d_{\mathbb{D}}(0,F(r\zeta))  \bigr \}
 \, = \,  \log |F'(\zeta)|,
$$
as desired.
\end{proof}

In view of the elementary estimate $\mu \le \eta$ and  Lemma \ref{geodesic-curvature-and-mobius-distortion}, the total M\"obius distortion and geodesic curvature are finite along $F([0, \zeta))$:
\begin{corollary}
\label{mobius-distortion}
Suppose $F$ is a holomorphic self-map of the disk with $F(0) = 0$. If $F$ has an angular derivative at $\zeta \in \partial \mathbb{D}$, then
$$
\int_0^\zeta \mu \, d\rho \le \log |F'(\zeta)|.
$$
and
$$
\int_0^\zeta \min \bigl (1, \kappa_{F([0,\zeta))}(F(z)) \bigr ) \, d\rho(z) \lesssim \log |F'(\zeta)|.
$$
\end{corollary}

Below, we will use the following lemma which follows from compactness:

\begin{lemma}
\label{continuity-of-the-radial-inefficiency}
There exists a $\delta > 0$ so that for any holomorphic self-map $F$ of the unit disk and any point $z \in \mathbb{D}$ with $d_{\mathbb{D}}(0,z) \ge 1$, we have
$$
\eta(z) < 0.1 \qquad \implies \qquad \eta(w) < 0.15, \quad w \in B_{\hyp}(z, \delta).
$$
\end{lemma}

\subsection{Bounding the radial inclination}

To complete the proof of Theorem \ref{total-radial-distortion}, it remains to estimate the radial inclination.
We parametrize the radial geodesic $\gamma(t) = [0, \zeta)$ with respect to arclength. 
We break up $(\gamma(1),\gamma(\infty))$ into a union of thick and thin intervals. By a {\em thin} interval $(\gamma(p_i), \gamma(q_i)) \subset (\gamma(1),\gamma(\infty))$, we mean a maximal interval for which $\eta(\gamma(p_i)) < 0.1$ and $\eta(\gamma(q_i)) < 0.2$.
   The thick intervals are then defined as the connected components of the complement of the thin intervals.

 In view of Lemma \ref{continuity-of-the-radial-inefficiency}, the hyperbolic length of a thin interval is bounded from below. Therefore, by Lemma \ref{radial-inefficiency}, the number of thin intervals $n(\zeta) \lesssim \log |F'(\zeta)|$. As thin and thick intervals alternate, the number of thick intervals is also $\lesssim \log |F'(\zeta)|$. 

 \begin{proof}[Proof of Lemma \ref{radial-inclination}]
Since $\eta(t) \ge 0.1$ on any thick interval, by Lemma \ref{radial-inefficiency}, the sum of the hyperbolic lengths of the thick intervals is $\lesssim \log |F'(\zeta)|$, so that
$$
  \sum_{\gamma_i \thick} \int_{\gamma_i} \alpha \, d\rho \lesssim\, \log|F'(\zeta)|.
  $$
 From the definitions of the radial inclination and the radial inefficiency, it follows that on a thin interval $\alpha(t) \le |\arg(0.8 + 0.2i)| \approx 0.644 < 2\pi/3$, so that
Lemma \ref{inclination-bound} is applicable (see the remark below).
 Together with Corollary \ref{mobius-distortion}, this shows
  $$
  \sum_{\gamma_i \thin} \int_{\gamma_i} \alpha \, d\rho \, \lesssim \, n(\zeta) +   \sum_{i=1}^n \int_{\gamma_i} \kappa \, d\rho \,  \lesssim\, \log|F'(\zeta)|.
  $$
  The proof is complete.
  \end{proof}
  
  \begin{remark}
   Actually, one needs to be a bit more precise in the proof above since Lemma \ref{inclination-bound} is stated on the upper half-plane. On the unit disk, one can only apply Lemma \ref{inclination-bound} as long as one is working sufficiently close to the unit circle. As $\eta(z) < 0.2$ on a thin interval, $F(z)$ moves towards the unit circle at a definite rate, so after $O(1)$ time, Lemma  \ref{inclination-bound} will indeed be applicable. The waiting time contributes at most $O(1)$ to each integral $\int_{\gamma_i} \alpha \, d\rho$ over a thin interval $\gamma_i$.
   \end{remark}

\part{General Centered Inner Functions of Finite Lyapunov Exponent}
\label{part:general-inner-function}

In this part, $F$ will denote an arbitrary centered inner function of finite Lyapunov exponent, other than a rotation. 
In Section \ref{sec:mobius-linear-laminations}, we define the M\"obius and linear laminations $\widehat{X}_{\mob}$ and $\widehat{X}_{\lin}$ associated to $F$ and describe the geodesic and horocyclic flows on $\widehat{X}_{\lin}$. To be fair, the term ``lamination'' is not entirely accurate here as $\widehat{X}_{\lin}$ and $\widehat{X}_{\mob}$ may not locally be product sets.

In Section \ref{sec:area}, we construct a natural volume form $d\xi$ on $\widehat{X}$.
According to Theorem \ref{total-mass}, the total volume of $\widehat{X}$ is just the Lyapunov exponent of $F$. From the finiteness of volume, it follows
that iteration along almost every backward orbit is asymptotically M\"obius, i.e.~$\xi(\widehat{X} \setminus \widehat{X}_{\mob}) = 0$.
In Section \ref{sec:linear-structure}, we improve this to asymptotically linearity, i.e.~$\xi(\widehat{X} \setminus \widehat{X}_{\lin}) = 0$.

 In Section \ref{sec:geodesic-folation}, we study how the trajectories of the geodesic flow foliate 
$\widehat{\mathbb{D}}_{\lin}$ and conclude that the geodesic flow on $\widehat{X}_{\lin}$ is ergodic. Finally, in Section \ref{sec:orbit-counting}, we apply the ergodic theorem to a slight modification of the almost invariant function from Section \ref{sec:1c-orbit-counting} to prove Theorem \ref{main-thm2}.
 
\section{M\"obius and Linear Laminations}
\label{sec:mobius-linear-laminations}

For a general centered inner function, the lamination $\widehat{X} = \widehat{\mathbb{D}} \, / \, \widehat F$ defined in Section \ref{sec:fbp} has limited use. In this section, we describe two subsets
$$\widehat{X}_{\mob} = \widehat{\mathbb{D}}_{\mob} \, / \, \widehat F \qquad \text{and} \qquad \widehat{X}_{\lin} = \widehat{\mathbb{D}}_{\lin} \, / \, \widehat F,
$$
which we refer to as the M\"obius and linear laminations of $F$ respectively. Here, $\widehat{\mathbb{D}}_{\mob} \subset \widehat{\mathbb{D}}$ is the collection of inverse orbits ${\bf z}  = (z_{-n})_{n=0}^\infty$ on which backward iteration is asymptotically M\"obius:
$$
 \mu_{F^{\circ m}}(z_{-m-n}) \to 0, \qquad \text{as } m,n \to \infty,
$$
while $\widehat{\mathbb{D}}_{\lin} \subset \widehat{\mathbb{D}}$ consists of inverse orbits on which backward iteration is asymptotically linear:
$$
 \delta_{F^{\circ m}}(z_{-m-n}) \to 0, \qquad \text{as } m,n \to \infty.
$$
(As $\delta_{F^{\circ m}}(z_{-m-n})$ is small, $F^{\circ m}$ is close to a straight M\"obius transformation near $z_{-m-n}$. Asymptotic linearity follows from the fact that $|z_{-n}| \to 1$.) Since $\mu \le \delta$, it is clear that $\widehat{X}_{\lin} \subset \widehat{X}_{\mob} \subset \widehat{X}$. 

On the set ${\widehat X}_{\mob} \subset \widehat{X}$, one can define a leafwise hyperbolic Laplacian and study mixing properties of hyperbolic Brownian motion, but we will not pursue this here. On $\widehat{X}_{\lin} \subset \widehat{X}$, one can define geodesic and horocyclic flows as in Section \ref{sec:fbp}. 

\subsection{Rescaling along inverse orbits} 

Inspection shows that a backward orbit ${\bf z} = (z_{-n})_{n=0}^\infty \in \widehat{\mathbb{D}}$ belongs to ${\widehat{\mathbb{D}}}_{\mob}$ if and only if there exists a sequence of M\"obius transformations ${m}_{-N} \in \aut(\mathbb{D})$, $N \in\N$, with
$$
{m}_{-N}(0)=z_{-N}
$$ 
for which the sequence 
$$
(F^{\circ N} \circ {m}_{-N})_{N=0}^\infty
$$ 
converges uniformly on compact subsets of the unit disk as $N \to \infty$. In this case, we denote the limiting map by $F_{{\bf z}, 0}$. In fact, when $F_{{\bf z},0}$ exists, so does
$$
F_{{\bf z}, -n} = \lim_{N \to \infty} F^{\circ (N-n)} \circ {m}_{-N}, \qquad \text{for any }n \in \mathbb{N},
$$
and $(F_{{\bf z}, -n}(w))_{n=0}^\infty$ defines an inverse orbit in ${\widehat{\mathbb{D}}}_{\mob}$ for any $w \in \mathbb{D}$.

We say that a backward orbit ${\bf w} = (w_{-n})_{n=0}^\infty$ lies in the same leaf of ${\widehat{\mathbb{D}}}_{\mob}$ as ${\bf z} = (z_{-n})_{n=0}^\infty$ if
there is a $w \in \mathbb{D}$ such that 
$$
w_{-n} = F_{{\bf z}, -n}(w),
$$
for all integers $n \in \mathbb{N}$.

For a point $p \in \mathbb{D} \setminus 0$, we write $M_{p}$ for the conformal map from $\mathbb{H}$ to $\mathbb{D}$ which takes
$$
0 \to \frac{p}{|p|}, \qquad i \to p, \qquad \infty \to -\frac{p}{|p|}.
$$
Similarly, a backward orbit ${\bf z} = (z_{-n})_{n=0}^\infty \in \widehat{\mathbb{D}}$ belongs to ${\widehat{\mathbb{D}}}_{\lin}$ if and only if for some
 (and hence any) $n \in \mathbb{N}$, the sequence of rescaled iterates
$$
F^{\circ (N - n)} \circ M_{z_{-N}},
$$
converges uniformly on compact subsets of $\mathbb{H}$ as $N \to \infty$. We denote the limiting maps by 
$$
F_{{\bf z}, -n} := \lim_{N \to \infty} F^{\circ (N-n)} \circ M_{z_{-N}}.
$$
We partition ${\widehat{\mathbb{D}}}_{\lin}$ into a union of leaves analogously to ${\widehat{\mathbb{D}}}_{\mob}$.

\begin{lemma}
\label{when-are-two-orbits-in-the-same-leaf}
Two inverse orbits ${\bf z} = (z_{-n})_{n=0}^\infty$ and ${\bf z'} = (z_{-n}')_{n=0}^\infty$ in $\widehat{\mathbb{D}}_{\lin}$ belong to the same leaf $\mathcal L \subset \widehat{\mathbb{D}}_{\lin}$ if and only if $\(d_{\mathbb{D}}(z_{-n}, z'_{-n}))_{n=0}^\infty$ is uniformly bounded. In this  case, the leafwise hyperbolic distance
$$
d_{\mathcal L}({\bf z}, {\bf z'}) = \lim_{n \to \infty} d_{\mathbb{D}}(z_{-n}, z'_{-n}).
$$
\end{lemma}

We define the geodesic and horocyclic flows on ${\widehat{\mathbb{D}}}_{\lin}$ by the following formulas:
$$
g_t({\bf z})_{-n} := F_{{\bf z}, -n}(e^t \cdot i), \qquad t \in \mathbb{R}
$$
and
$$
h_s({\bf z})_{-n} := F_{{\bf z}, -n}(i+s), \qquad s \in \mathbb{R}.
$$
Clearly, the $(-n)$-th coordinates of geodesic trajectories which foliate the leaf $\mathcal L({\bf z}) \subset {\widehat{\mathbb{D}}}_{\lin}$ containing ${\bf z}$
are images of vertical geodesics $\{ w \in \mathbb{H} : \re w = x \}$ under $F_{{\bf z},-n}$.

The choices of basepoints $0 \in \mathbb{D}$ and $i \in \mathbb{H}$ in the definitions above are of course arbitrary.

\begin{remark}
For a general centered inner function, the laminations $\widehat{X}_{\mob}$ and $\widehat{X}_{\lin}$ could be empty.
For instance, there exists a centered inner function $F$ whose critical set forms a {\em net}\/, i.e.~there exists an $R > 0$ so that any point in the unit disk is within hyperbolic distance $R$ of a critical point. However, in view of Jensen's formula, $F$ does not have a finite Lyapunov exponent.
\end{remark}

\subsection{Cumulative distortion}
\label{sec:cumulative-distortion}

We now introduce some notions which allow us to check whether an inverse orbit ${\bf z}$ lies in $\widehat{\mathbb{D}}_{\mob}$ or $\widehat{\mathbb{D}}_{\lin}$.

We denote the {\em cumulative hyperbolic expansion factor} by $$E(w, z) = \| (F^{\circ n})'(w) \|_{\hyp}^{-1}$$
if $F^{\circ n}(w) = z$ and
$$E({\bf w}) = E({\bf w}, z) = \lim_{n \to \infty} \| (F^{\circ n})'(w_{-n}) \|_{\hyp}^{-1}$$ if
${\bf w} = (w_{-n})_{n=0}^\infty \in \widehat{\mathbb{D}}$ is an inverse orbit with $w_0 = z$.  It is easy to see that ${\bf w} \in \widehat{\mathbb{D}}_{\mob}$ if and only if
$E({\bf w}) < \infty$.

We denote the {\em cumulative linear distortion} along an inverse orbit ${\bf z} \in \widehat{\mathbb{D}}$ by
\begin{equation}
\label{eq:cumulative-linear-distortion}
\widehat{\delta}_F({\bf z}) := \sum_{n=1}^{\infty} \delta_F(z_{-n}).
\end{equation}
%To estimate the cumulative linear distortion, we use the following lemma:

\begin{lemma}
\label{cumulative-distortion}
Suppose $F$ and $G$ are holomorphic self-maps of the unit disk. For a point
$a \in \mathbb{D}$ such that $a, G(a), F(G(a)) \ne 0$, we have
$$
\delta_{F \circ G}(a) \le \delta_{F}(G(a)) +  \delta_{F}(a).
$$
In particular, if $a, F(a), \dots, F^{\circ n-1}(a) \ne 0$, then
$$
\delta_{F^{\circ n}}(a) \le  \sum_{k=0}^{n-1} \delta_F(F^{\circ k}(a)).
$$
\end{lemma}

\begin{proof}
Notice that if $p, q \in \mathbb{D}$, then
\begin{align*}
|1 - pq|= |1 - p| + |p - pq|\le |1 - p| + |1 - q|.
\end{align*}
The lemma follows from the above identity, with
$$
p = \frac{G_* v_{\rad}(a)}{v_{\rad}(G(a))} \qquad \text{and} \qquad
q = \frac{F_* v_{\rad}(G(a))}{v_{\rad}(F(G(a)))},
$$
as $\delta_{G}(a) = |1-p|$, $\delta_{F}(G(a)) = |1-q|$ and $\delta_{F \circ G}(a) = |1-pq|$.
\end{proof}
From the lemma above, it is clear that if $\widehat{\delta}_F({\bf z}) < \infty$, then ${\bf z} \in \widehat{\mathbb{D}}_{\lin}$.

\section{Area on the Lamination}
\label{sec:area}

Throughout this section, $F$ will be a centered inner function with finite Lyapunov exponent.
For a measurable set $A$ compactly contained in the unit disk, we write $\widehat{A}$ for the collection of inverse orbits ${\bf z}$ with $z_0 \in A$.
We define an $\widehat{F}$-invariant measure $\xi$ on $\widehat{\mathbb{D}}$ by specifying it on sets of the form $\widehat{A} \subset \widehat{\mathbb{D}}$ in a consistent manner:
\begin{equation}
\label{eq:xi-def-a}
\xi(\widehat{A}) = \lim_{n \to \infty}  \frac{1}{2\pi} \int_{F^{-n}(A)}  \log \frac{1}{|z|} \, dA_{\hyp}(z).
\end{equation}
In order to show that the limit in (\ref{eq:xi-def-a}) exists, we check that the numbers
$$
\int_{F^{-n}(A)} \log \frac{1}{|z|} \, dA_{\hyp}
$$
are increasing and uniformly bounded above. This follows from Lemma \ref{monotonicity-lemma}
and Theorem \ref{total-mass} below:

 \begin{lemma}
 \label{monotonicity-lemma}
For a measurable subset $E$ of the unit disk,
$$
\int_{F^{-1}(E)} \log \frac{1}{|z|} \, dA_{\hyp}  \ge \int_{E} \log \frac{1}{|z|} \, dA_{\hyp}.
$$
\end{lemma}

\begin{proof}
A change of variables shows that
$$
\int_{F^{-1}(E)} \log \frac{1}{|w|}  \, dA_{\hyp(w)} = \int_{E} \biggl \{ \sum_{F(w)=z} \| F'(w) \|_{\hyp}^{-2} \, \log \frac{1}{|w|} \biggr \} dA_{\hyp}(z)
$$
By the Schwarz lemma and Lemma \ref{sum-of-heights}, this is
$$
\ge \int_{E} \biggl \{ \sum_{F(w)=z} \log \frac{1}{|w|} \biggr \} dA_{\hyp}(z) =  \int_{E} \log \frac{1}{|z|} \, dA_{\hyp}
$$
as desired.
\end{proof}

\begin{theorem}
\label{total-mass}
The total mass  $\xi(\widehat{X}) = \int_{S^1} \log|F'(z)| dm$.
\end{theorem}

\begin{proof}
Since $F$ has an angular derivative a.e.~on the unit circle, for any $\varepsilon > 0$, there is a Borel set $A_\varepsilon \subset S^1$ with $m(A_\varepsilon) \ge 1-\varepsilon$ and an $0 < r_0 = r_0(\varepsilon) < 1$ so that
$$
|F(re^{i\theta}) - F(e^{i\theta}) - (1-r) F'(e^{i\theta})| < \varepsilon(1-r),
$$
for all $e^{i\theta} \in A_\varepsilon$ and $r\in [r_0, 1)$. Consider the set
$$
\widetilde{A}_\varepsilon = \biggl \{re^{i\theta} \in \mathbb{D} \, : \,  e^{i\theta} \in A_\varepsilon, \,  r_0 \le r \le 1 - \frac{1-r_0}{|F'(e^{i\theta})| - \varepsilon} \biggr \},
$$
where we use the convention that $|F'(e^{i\theta})| = \infty$ if the angular derivative does not exist.
By construction, the image $F(\widetilde{A}_\varepsilon)$ is contained in the ball $B(0, r_0)$ so that $\widetilde{A}_\varepsilon$ does not intersect any of its forward iterates. Therefore, by Lemma  \ref{monotonicity-lemma},
$$
\xi(\widehat{X}) \,  \ge \,
 \frac{1}{2\pi} \int_{{\widetilde A}_\varepsilon}  \log \frac{1}{|z|} \, dA_{\hyp} \, \ge \,  \int_{A_\varepsilon} (\log|F'(e^{i\theta})| - \varepsilon)
 dm.
$$
Taking $\varepsilon \to 0$ proves the lower bound.

For the upper bound, suppose that $E$ is a subset of the unit disk which is disjoint from its backward iterates. We want to show that
$$
 \frac{1}{2\pi} \int_{E} \log \frac{1}{|z|} \, dA_{\hyp} \le \int_{S^1} \log|F'(z)| dm.
$$
Truncating $E$ if necessary, we may assume that $E$ is contained in a ball $B(0, r_0)$ for some $0 < r_0 < 1$. Consider the set $E^* = F^{-1}(B(0,r_0)) \setminus B(0,r_0)$. By construction,
$$
 \int_{E}  \log \frac{1}{|z|} \, dA_{\hyp} \le  \int_{E^*} \log \frac{1}{|z|} \, dA_{\hyp}.
$$
By Lemma \ref{ac-lemma}, the set $E^*$ is contained in the union of
$$
E^*_1 = \biggl \{re^{i\theta} \in \mathbb{D} \, : \,  e^{i\theta} \in A_\varepsilon, \,  r_0 \le r \le 1 - \frac{1-r_0}{|F'(e^{i\theta})| + \varepsilon} \biggr \}
$$
and
$$
E^*_2 = \biggl \{re^{i\theta} \in \mathbb{D}  \, : \,  e^{i\theta} \notin A_\varepsilon, \, r_0 \le r \le 1 - \frac{1-r_0}{4|F'(e^{i\theta})|} \biggr \},
$$
so that
$$
\frac{1}{2\pi} \int_{E^*} \log \frac{1}{|z|} \, dA_{\hyp} \le \frac{1}{2\pi} \int_{E_1^*} \log \frac{1}{|z|} \, dA_{\hyp} + \frac{1}{2\pi} \int_{E_2^*} \log \frac{1}{|z|} \, dA_{\hyp}.
 $$
The theorem follows after taking $\varepsilon \to 0$.
\end{proof}

\subsection{M\"obius structure}

We will deduce the following theorem from the finiteness of the area of the Riemann surface lamination:

\begin{theorem}[M\"obius structure]
\label{mobius-structure}
Backward iteration along $\xi$ a.e.~inverse orbit is asymptotically close to a M\"obius transformation, i.e.~$\xi(\widehat{{\mathbb{D}}} \setminus \widehat{\mathbb{D}}_{\mob}) = 0$.
\end{theorem}

%We first describe an important class of balls that will be used repeatedly in the arguments below. 
Suppose $z \in \mathbb{D}$ is a point in the unit disk, which is not contained in the forward orbit of an exceptional point so that $\overline{c}_z$ is a probability measure on $T(z)$, see Section \ref{sec:transversals} for the relevant definitions. We fix a constant $0 < \gamma \le 1$ for which Lemma \ref{minimal-translation} holds. Then for any ball  $\mathscr B = B_{\hyp}(z, \gamma)$
with $d_{\mathbb{D}}(0,z) > 1+\gamma$, the natural projection from $\widehat{\mathbb{D}} \to \widehat{X}$ is injective on $\widehat{\mathscr B}$.

\begin{proof}
From the definition of the measure $\xi$, we have
$$
\xi(\widehat{\mathscr B}) = \int_{\mathscr B} \Psi(z') \, \log \frac{1}{|z'|} \, dA_{\hyp}(z')
$$
where
\begin{align*}
\Psi(z') & = \lim_{n \to \infty} \sum_{F^{\circ n}(w') = z'}   \log \frac{1}{|w'|} \cdot \| (F^{\circ n})'(w') \|_{\hyp}^{-2} \\
& = \int_{T(z')} E({\bf w'}, z')^2 \, d\overline{c}_{ z'}({\bf w'})
\end{align*}
is the {\em average area expansion factor}\/. Since $\xi$ is a finite measure, $\Psi(z') < \infty$ for Lebesgue a.e.~$z' \in \mathscr B$
and $E({\bf w'}, z') < \infty$ for $\xi$ a.e~${\bf w'} \in \widehat{\mathscr B}$.
As discussed in Section \ref{sec:mobius-distortion}, this implies that ${\bf w'} \in \widehat{\mathbb{D}}_{\mob}$. 

The theorem follows from the observation that countably many sets of the form
$\widehat{\mathscr B}$ cover $\widehat{X}$.
\end{proof}

\subsection{M\"obius decomposition theorem}

We say that a repeated pre-image $w$ of $z$ is $\varepsilon$-(M\"obius good) 
 if the hyperbolic expansion factor
$$
1 \le \| (F^{\circ n})'(w) \|_{\hyp}^{-1} < 1+\varepsilon, \qquad \text{where }F^{\circ n}(w) = z.
$$
In view of Lemma \ref{small-mobius-distortion}, when $\varepsilon > 0$ is sufficiently small, the connected component
of $F^{-n}(\mathscr B)$ containing $w$ maps conformally onto $\mathscr B$ under the dynamics of $F$. Naturally, we call it $\mathscr B_w$.
By shrinking $\varepsilon > 0$ further, we may assume that
$$
1 \le \| (F^{\circ n})'(q) \|_{\hyp}^{-1} < 2, \qquad \text{for any } q\in \mathscr B_w.
$$
Similarly, we say that an inverse orbit ${\bf w} \in T(z)$ is  $\varepsilon$-(M\"obius good) if 
$$
1 \le \| (F^{\circ n})'(w_{-n}) \|_{\hyp}^{-1} < 1+\varepsilon,
$$
for any integer $n \in \mathbb{N}$. We define
$$
\widehat{\mathscr B}_{\Mgood}:= \bigcup_{{\bf w} \in T_{\Mgood}(z)} \mathscr B_{\bf w},
$$
where ${\bf w}$ ranges over $T_{\Mgood}(z)$, the set of $\varepsilon$-(M\"obius good) inverse orbits with $w_0 = z$.
On $\widehat{\mathscr B}_{\Mgood}$, the measure $\xi$ is comparable to the product measure
\begin{equation}
\label{eq:xi-local-form}
\underbrace{\log \frac{1}{|q|} \, dA_{\hyp}(q)}_{\text{on }\mathscr B} \ \times \ \underbrace{c_z}_{\text{on }T_{\Mgood}(z)}.
\end{equation}

\begin{remark}
In the one component case, the Riemann surface lamination $\widehat{X}$ is locally a product space. 
The ``charts'' $\widehat{\mathscr B}_{\Mgood}$ may be viewed as substitutes of the product sets $\widehat{\mathscr B}$ from the one component setting.
\end{remark}

We say that a point $z \in \mathbb{D}$ is $\varepsilon$-(M\"obius nice) if most inverse branches ${\bf w} \in T(z)$ are $\varepsilon$-(M\"obius good):
 $$
\overline{c}_z(T_{\Mgood}(z)) > 1 - \varepsilon,
 $$
which is the same as asking that
$$
\sum_{\substack{F^{\circ n}(w) = z, \, n \ge 0 \\ w \, \Mgood}} \log\frac{1}{|w|}> (1 - \varepsilon) \log \frac{1}{|z|},
$$
for any $n \ge 0$.

\begin{theorem}[M\"obius decomposition theorem]
\label{mobius-decomposition-theorem}
For a centered inner function  $F$ with finite Lyapunov exponent, the following two assertions hold:

{\em (a)} For any $\varepsilon > 0$ and almost every point $z \in \mathbb{D}$, there exists an $n \ge 0$ so that
\begin{equation}
\label{eq:mobius-decomposition-theorem}
\sum_{\substack{F^{\circ n}(w) = z \\ w \text{\,is\,}\Mnice }} \log \frac{1}{|w|} > (1-\varepsilon) \cdot \log \frac{1}{|z|}.
\end{equation}

{\em (b)} For any $\varepsilon > 0$, one can find finitely many $\varepsilon$-(M\"obius nice) points $z_1, z_2, \dots, z_N$, so that the sets
$$\widehat{\mathscr B}_{i, \, \Mgood}, \qquad i=1, 2, \dots, N,$$ cover $\widehat{X}$ up to $\xi$-measure $\varepsilon$, i.e.~
$$
\xi \biggl ( \widehat{X} \setminus \bigcup_{i=1}^N \widehat{\mathscr B}_{i, \, \Mgood} \biggr ) < \varepsilon.
$$
\end{theorem}

\begin{proof}
(a) Suppose ${\bf w} \in T(z)$ is a backward orbit. By the Schwarz lemma, the numbers $E(w_{-n}, z)$ increase to $E({\bf w}, z)$ as $n \to \infty$, which may be infinite. Consequently, if (\ref{eq:mobius-decomposition-theorem}) fails at a point $z \in \mathbb{D}$ for all $n \ge 0$, then for at least $\overline{c}_z$ measure $\varepsilon$ backward orbits ${\bf w} \in T(z)$, the area expansion factor
$E({\bf w}, z) = \infty$. In this case, the average area expansion factor $\Psi(z) = \infty$. However, in the proof of Theorem \ref{mobius-structure}, we saw that $\Psi(z) < \infty$ a.e., so (\ref{eq:mobius-decomposition-theorem}) can only fail on a set of Lebesgue measure zero.

(b) 
If $z \in \mathbb{D}$ is not $\varepsilon$-(M\"obius nice), then
$$
\xi \bigl (\widehat{B}_{\hyp}(z, \gamma) \bigr ) \, > \, \Theta \int_{B_{\hyp}(z, \gamma)}  \log \frac{1}{|z|} \, dA_{\hyp},
$$
for some $\Theta(\varepsilon) > 1$.
An examination of the proof of Theorem \ref{total-mass} shows that for any $\eta > 0$,
$$
\int_{E^*} \log \frac{1}{|z|} \, dA_{\hyp} \ge (1 - \eta) \int_{S^1} \log|F'(z)| dm,
$$
where $E^* = F^{-1}(B(0,r_0)) \setminus B(0,r_0)$ and $0 < r_0 < 1$ is sufficiently close to 1. 

Therefore, by asking for $r_0$ to be sufficiently close to 1, we can make the $\log \frac{1}{|z|} \, dA_{\hyp}(z)$ area of 
$$S = \{ z \in E^* : z \text{ is not }\!\Mnice \}$$
as small as we wish. We may choose finitely many $\varepsilon$-(M\"obius nice) points $\{z_i\}_{i=1}^N$ in $E^* \setminus S$ such that the balls
of hyperbolic radius $\gamma$ centered at these points cover $E^* \setminus S$ up to small  $\log \frac{1}{|z|} \, dA_{\hyp}(z)$ measure.
Consequently, the sets $$\widehat B_{\hyp}(z_i, \gamma), \qquad i = 1, 2, \dots, N,$$ cover $\widehat{X}$ up to small measure.
\end{proof}

\section{Linear structure}
\label{sec:linear-structure}

In this section, we show that backward iteration along almost every inverse orbit is asymptotically linear:

\begin{theorem}[Linear structure]
\label{nearly-linear}
For $\xi$ a.e.~inverse orbit  ${\bf z} = (z_{-i})_{i=0}^\infty \in \widehat{\mathbb{D}}$,  the cumulative linear distortion $\widehat{\delta}(\bf z) < \infty$.
Consequently,  $\xi(\widehat{\mathbb{D}} \setminus \widehat{\mathbb{D}}_{\lin}) = 0$.
\end{theorem}

\begin{proof}
In view of Theorem \ref{mobius-decomposition-theorem}, we may show that $\widehat{\delta}(\bf w) < \infty$ for a.e.~inverse orbit ${\bf w} \in \widehat{\mathscr B}_{\Mgood}$ where $\mathscr B = B_{\hyp}(z, \gamma)$ is a ball centered at an $\varepsilon$-(M\"obius nice) point $z \in \mathbb{D}$. 

\allowdisplaybreaks

Let $ \widetilde{\mathscr B}_{\Mgood} \subset \mathbb{D}$ be the union of topological disks $\mathscr B_w$, where $w$ ranges over the repeated pre-images of $z$ with $E(w, z) < 1 + \varepsilon$. We may assume that $\varepsilon > 0$ is sufficiently small so that $E(w, z) < 1 + \varepsilon$ implies that
 $E(\tilde w, \tilde z) < 2$ for any $\tilde{w} \in \mathscr B_w$ and  $\tilde{z} \in \mathscr B$.
  By Theorem \ref{total-radial-distortion}, we have
\begin{align*}
\label{eq:total-radial-distortion3}
\int_{\widehat{\mathscr B}_{\Mgood}}{\widehat \delta}({\bf w}) d\xi 
& \le 4 \int_{\widetilde{\mathscr B}_{\Mgood}} \delta(w) \cdot \log \frac{1}{|w|} \, dA_{\hyp}(w) \\
& \lesssim  \int_{\mathbb{D}}\delta(z) \cdot \log \frac{1}{|w|} \, dA_{\hyp}(w) \\
& \lesssim \int_{\partial \mathbb{D}} \log |F'(re^{i\theta})| dm \\
& < \infty,
\end{align*}
which shows that $\widehat{\delta}({\bf w})$ is finite $\xi$ almost everywhere. By the discussion in Section \ref{sec:cumulative-distortion}, $\xi$ gives full mass to $\widehat{\mathbb{D}}_{\lin} \subset \widehat{\mathbb{D}}$.
\end{proof}

\begin{corollary}
\label{nearly-linear2}
Let ${\bf z} = (z_{-i})_{i=0}^\infty \in \widehat{\mathbb{D}}$ be a generic backwards orbit. For any $R, \varepsilon > 0$, there exists an $n_0 = n_0({\bf z}, \varepsilon, R) > 0$ sufficiently large so that for any $ n > m \ge n_0$, the inverse branch $g_{m,n}$ of
$
F^{m-n} : z_{-m} \to z_{-n}
$
is well-defined on  $B_{\hyp}(z_{-m}, R)$, where it is within hyperbolic distance $O(\varepsilon)$ of the linear map $\ell_{m,n} \in \aut(\mathbb{C})$ which takes
$$z_{-m} \to z_{-n} \qquad \text{and} \qquad \frac{z_{-m}}{|z_{-m}|} \to \frac{z_{-n}}{|z_{-n}|}.$$
\end{corollary}

\begin{proof}
For an inverse orbit ${\bf z} = (z_{-i})_{i=0}^\infty$ with $\widehat{\delta}({\bf z}) < \infty$, we choose $n_0 = n_0({\bf z}, \varepsilon, R)$ sufficiently large so that $$\sum_{n=n_0+1}^{\infty} \delta(z_{-i}) < \varepsilon \qquad \text{and} \qquad 1-|z_{-n_0}| < \varepsilon/e^R.$$
In view of Lemma \ref{small-mobius-distortion}, $g_{m,n}$ is well-defined on $B_{\hyp}(z_{-m}, R)$, where it is within hyperbolic distance $O(\varepsilon)$ of the straight M\"obius transformation in $\aut(\mathbb{D})$ which takes $z_{-m} \to z_{-n}$. The second condition $1-|z_{-n_0}| < \varepsilon/e^R$ ensures that $g_{m,n}$ is within hyperbolic distance $O(\varepsilon)$ of $\ell_{m,n}$ on $B_{\hyp}(z_{-m}, R)$.
\end{proof}

A similar argument involving Lemma \ref{small-linear-distortion2} shows:

\begin{corollary}
\label{nearly-linear3}
In the above corollary, we can select $n_0 = n_0({\bf z}, \varepsilon, R) > 0$ sufficiently large so that
\begin{equation*}
\label{eq:nearly-linear4}
\int_{g_{m,n}(E)}  \log \frac{1}{|z|} \, dA_{\hyp}(z) \sim_\varepsilon \int_{E}  \log \frac{1}{|z|} \, dA_{\hyp}(z),
\end{equation*}
for any measurable set $E \subset B_{\hyp}(z_{-m}, R)$, where the notation $A \sim_\varepsilon B$ indicates that $(1- C\varepsilon)A \le B \le (1 + C\varepsilon)A$ for some constant $C > 0$, which depends only on $R$. 
\end{corollary}

\subsection{Linear decomposition theorem}

We say that a point $z \in \mathbb{D}$ is $\varepsilon$-(linear nice) if $1-|z| < \varepsilon/e^\gamma$ and for {\em most} inverse branches, backward iteration is close to a linear mapping:
$$
c_z \bigl (\{ {\bf w} \in T(z) : \widehat{\delta}({\bf w}) < \varepsilon \}  \bigr ) > (1-\varepsilon) \cdot \log \frac{1}{|z|},
$$
where $\widehat{\delta}({\bf w})$ is the cumulative linear distortion defined in Section \ref{sec:cumulative-distortion}. 

\begin{theorem}[Linear decomposition theorem]
\label{linear-decomposition-theorem}
{\em (a)} For any $\varepsilon > 0$ and almost every point $z \in \mathbb{D}$, there exists an $n \ge 0$ so that
\begin{equation}
\label{eq:linear-decomposition-theorem}
\sum_{\substack{F^{\circ n}(w) = z \\ w \text{\,is\,}\Lnice }} \log \frac{1}{|w|} > (1-\varepsilon) \cdot \log \frac{1}{|z|}.
\end{equation}

{\em (b)}
For any $\varepsilon > 0$, one can find finitely many $\varepsilon$-(linear nice) points $z_1, z_2, \dots, z_N$ so that
$$
\xi \biggl ( \widehat{X} \setminus \bigcup_{i=1}^N \widehat{\mathscr B}_{i, \, \Lgood} \biggr ) < \varepsilon.
$$
\end{theorem}

\begin{proof}
(a) For a point  $z' \in \mathbb{D}$, let  $\Delta_{z'}$ denote the set of inverse orbits ${\bf w'} \in T(z')$
for which the cumulative linear distortion $\widehat{\delta}({\bf w'}) = \infty$.
If (\ref{eq:linear-decomposition-theorem}) fails at $z' \in \mathbb{D}$, then
$\overline{c}_{z'}(\Delta_{z'}) \ge \varepsilon$. 

For the sake of contradiction, assume that (\ref{eq:linear-decomposition-theorem}) fails on a set of positive Lebesgue measure $A$ in the unit disk. 
However, by the Schwarz lemma, this would imply that
\begin{align*}
\int_{\widehat{A}} \chi_{\{{\bf w}:\, \widehat{\delta}({\bf w}) = \infty \}} d\xi({\bf w})
& = \int_{A} \int_{T(z')} \chi_{\{{\bf w}:\, \widehat{\delta}({\bf w}) = \infty \}} E({\bf w'}, z')^2 \, d\overline{c}_{ z'}({\bf w'}) \log \frac{1}{|z'|} dA_{\hyp}(z') \\
& \ge \int_{A} \overline{c}_{ z'}(\Delta_{z'}) \log \frac{1}{|z'|} dA_{\hyp}(z') \\
& > 0,
\end{align*}
contradicting Theorem \ref{nearly-linear} which says that $\widehat{\delta}({\bf z'}) < \infty$ for Lebesgue a.e.~${\bf z'} \in \widehat{\mathbb{D}}$.

(b) The proof is similar to that of part (b) in Theorem \ref{mobius-decomposition-theorem}. 
\end{proof}

\section{The Geodesic Foliation Theorem}
\label{sec:geodesic-folation}

In this section, we show the following theorem which describes the structure of geodesic trajectories in $\widehat{\mathbb{D}}_{\lin}$\,:

\begin{theorem}
\label{geodesic-foliation-theorem}
{\em (i)} For $\xi$ a.e.~backward orbit ${\bf z} \in \widehat{\mathbb{D}}_{\lin}$ and $n \ge 0$, the limit
$$
\zeta_{-n}({\bf z}) :=  \lim_{t \to \infty} (g_{-t}({\bf z}))_{-n}
$$
exists and $(\zeta_{-n}({\bf z}))$ belongs to the solenoid. 

{\em (ii)} Let $\gamma(t) = (g_{-t}({\bf z}))_{0}$. If $\overline{\gamma}(t)$ is the radial geodesic that connects $0$ with $\zeta_0 = \zeta_0({\bf z})$ parametrized with respect to unit hyperbolic speed, then
$$
\frac{1}{T} \int_0^T \min \bigl \{ 1, d_{\mathbb{D}} ( \gamma(t), \, \overline{\gamma}(t_0 + t)) \bigr \} dt \to 0, \qquad \text{as }T \to \infty,
$$
for some offset $t_0 \in \mathbb{R}$ depending on ${\bf z}$.

{\em (iii)} For $\widehat{m}$ a.e.~${\bf x} \in  \widehat{S^1}$, there exists a unique backward orbit in 
$\widehat{\mathbb{D}}_{\lin}$ that lands at
${\bf x}$.

{\em (iv)} If $E \subset  \widehat{S^1}$ has $\widehat{m}$ measure zero, then $\zeta^{-1}(E) \subset {\widehat{\mathbb{D}}}_{\lin}$ has $\xi$ measure zero.
 \end{theorem}
 
 As a consequence, we deduce that the geodesic flow is ergodic:
 
 \begin{corollary}
The geodesic flow on the Riemann surface lamination $\widehat{X}_{\lin}$ is ergodic.
\end{corollary}

\begin{proof}
Suppose $A \subset \widehat{X}_{\lin}$ is a $g_t$-invariant set. Lifting to $\widehat{\mathbb{D}}_{\lin}$, we get a $(g_t, \widehat{F})$-invariant set $\widetilde{A}$, which is a necessarily a union of geodesic trajectories. The endpoints of these trajectories under the backward geodesic flow form an $ \widehat{F}$-invariant set $\zeta_0(\widetilde{A})$ in the solenoid. Since the action of $\widehat{F}$ on the solenoid is ergodic, either $\zeta_0(\widetilde{A})$ or its complement has $\widehat{m}$ measure 0. By Theorem \ref{geodesic-foliation-theorem}(iv), either $\widetilde{A}$ or its complement has $\xi$ measure 0, and thus the same is true of $A$.
\end{proof}

\subsection{Trajectories land on the solenoid} 
\label{sec:trajectories-land-on-the-solenoid}

For $0 < r < 1$, we define the function $\widehat{\delta}_r: \widehat{X}_{\lin} \to \mathbb{R}$ by
$$
\widehat{\delta}_r({\bf z}) := \max \biggl \{ 1, \, \sum \delta(z_{-k}) \biggr \},
$$
where we sum over the part of the inverse orbit contained in the annulus $A(0; r, 1)$. For any $0 < r < 1$, the function
 $\widehat{\delta}_r({\bf z})$ belongs to $L^2(\widehat{X}_{\lin})$, and the functions $\widehat{\delta}_r({\bf z})$ decrease pointwise a.e.~to 0 as $r \to 1$.

By the ergodic theorem for invariant measures,
for $\xi$ a.e.~${\bf z} \in \widehat{X}_{\lin}$, the backward time average
$$
\widehat{\delta}_{r, -}({\bf z}) := \lim_{T \to \infty} \frac{1}{T} \int_0^T \widehat{\delta}_r(g_{-t}({\bf z})) dt
$$
is the orthogonal projection of $\widehat{\delta}_r$ onto the subspace of $g_t$-invariant functions in $L^2(\widehat{X}_{\lin})$. This implies that for $\xi$ a.e.~${\bf z} \in \widehat{X}_{\lin}$, we have
$$
 \lim_{T \to \infty} \frac{1}{T} \int_0^T \widehat{\delta}(g_{-t}({\bf z})) = 0,
$$
which implies (i) and (ii) by Theorem \ref{weak-shadowing-in-H}.

\subsection{Uniqueness}

Suppose ${\bf z}, {\bf z'} \in \widehat{\mathbb{D}}_{\lin}$ are two generic inverse orbits with respect to the measure $\xi$ for which $\zeta({\bf z}) = \zeta({\bf z'})$.
By part (ii), we know that for each $n \ge 0$, the trajectories $g_{-t}({\bf z})_{-n}$ and $g_{-t}({\bf z'})_{-n}$ both weakly shadow the same radial ray $[0, \zeta_{-n}]$. By
Lemma \ref{when-are-two-orbits-in-the-same-leaf}, the trajectories $g_{-t}({\bf z})$ and $g_{-t}({\bf z'})$ belong to the same leaf, which means that there exists a  vertical geodesic 
$$V_{\xi'} = \{ z \in \mathbb{H} : \re x = \xi' \} \subset \mathbb{H}$$ so that
$\{ g_{t}({\bf z'})_{-n} : t \in \mathbb{R} \} = F_{{\bf z}, -n}(V_{\xi'})$. Weak shadowing forces $\xi' = 0$, i.e.~${\bf z}$ and ${\bf z'}$ belong to the same geodesic trajectory, which proves the uniqueness statement in (iii).

\subsection{Rescaling limits and measures}

A set
$
A \subset \widehat{\mathscr B}_{\Lgood}
$
is naturally decomposed as a union of slices: $$A = \bigcup_{{\bf z} \in T_{\Lgood}(z)} A_{\bf z},$$ with the slice $A_{\bf z} \subset \mathscr B_{\bf z}$ consisting of inverse orbits ${\bf w}$ which follow {\bf z}, i.e.~$w_{-n}$ lies in the same connected component of $F^{-n}(\mathscr B)$ as $z_{-n}$ for any $n \in \mathbb{N}$. 

Via rescaling maps, we may view the slices of $A$ as subsets of the upper half-plane. More precisely, for ${\bf z} \in T_{\Lgood}(z)$, we may define the sets
$$
A_{\bf z}^{*} \, \subset \, \mathscr B_{\bf z}^{*} \, = \, F_{{\bf z}, 0}^{-1}(\mathscr{B}) \, \subset \, \mathbb{H}.
$$
\begin{theorem} 
\label{novel-formulas-for-measure}
The following equalities hold:
$$
\xi(A) = \int_{T_{\Lgood}(z)} \biggl \{  \int_{A_{\bf z}^*} \frac{dA(w)}{\im w} \biggr \} dc_z
$$
and
$$
\widehat{m}(\zeta(A)) =  \int_{T_{\Lgood}(z)} \ell (\Pi_{\mathbb{H} \to \mathbb{R}}(A_{\bf z}^*)) dc_z,
$$
where $\Pi_{\mathbb{H} \to \mathbb{R}}$ is the orthogonal projection onto the real line and $\ell$ is the Lebesgue measure on the real line.
\end{theorem}

The proof of the above theorem is somewhat involved and will be given in Appendix \ref{sec:integrating-over-leaves}.

\subsection{Abundance of landing points} 

We now show that the landing points of backward trajectories of the geodesic flow cover a positive $\widehat{m}$ measure of the solenoid $\widehat{S^1}$.
Since $\widehat{m}$ is ergodic with respect to the action of $\widehat{F}$, it will then follow that landing points of backward trajectories cover the solenoid up to measure zero, proving the existence statement in (iii).

For this purpose, we take $A = \widehat{\mathscr B}_{\Lgood}$ in Theorem \ref{novel-formulas-for-measure}. By the Schwarz lemma, each $A^*_{\bf z}$ with ${\bf z} \in T_{\Lgood}(z)$ contains the ball
$B_{\hyp}^{\mathbb{H}}(i, \gamma)$, while by $\varepsilon$-linearity, $A^*_{\bf z}$ is contained in the larger ball $B_{\hyp}^{\mathbb{H}}(i, 2 \gamma)$.  Consequently,
$$
\widehat{m}(\zeta(A)) \, = \,  \int_{T_{\Lgood}(z)} \ell(\Pi_{\mathbb{H} \to \mathbb{R}}(A^*_{\bf z})) dc_z \, \gtrsim \, c_z(T_{\Lgood}(z)),
$$
which is certainly positive if $z \in \mathbb{D}$ is $\varepsilon$-(linear nice).

\subsection{Non-singularity}

Finally, we show that if a set $A \subset \widehat{\mathbb{D}}$ has positive $\xi$ measure, then its projection $\zeta(A)$ to the solenoid has
positive $\widehat{m}$ measure. As the intersection of $A$ with some set of the form $\widehat{\mathscr B}_{\Lgood}$ has positive $\xi$ measure, we may assume that $A$ is contained in a single $\widehat{\mathscr B}_{\Lgood}$. Since
$$
\int_{K} \frac{dA(w)}{\im w} \lesssim \ell(\Pi_{\mathbb{H} \to \mathbb{R}}(K)),
$$
for any measurable set $K \subset B_{\hyp}^{\mathbb{H}}(i, 2 \gamma) \subset \mathbb{H}$, we have
$$
\xi(A) \lesssim \widehat{m}(\zeta(A)),
$$
so $\widehat{m}(\zeta(A)) > 0$ as well, which proves (iv).

%%%%
%%
%%%
 
\section{Orbit Counting}
\label{sec:orbit-counting}

In this section, we prove Theorem \ref{main-thm} on averaged orbit counting for centered inner functions of finite Lyapunov exponent.

\begin{theorem}
\label{main-thm-2a}
Let $F$ be an inner function of finite Lyapunov exponent with $F(0) = 0$ for which the geodesic flow  is ergodic on the Riemann surface lamination $\widehat{X}_{\lin}$. Suppose $z \in \mathbb{D} \setminus \{ 0 \}$ lies outside a set of measure zero. Then,
\begin{equation}
\label{eq:main-thm-2a}
\lim_{R\to+\infty}
\frac{1}{R} \int_0^R \frac{\mathcal N(z, S)}{e^S} dS 
=\frac{1}{2} \log \frac{1}{|z|} \cdot \frac{1}{\int_{\partial \mathbb{D}} \log |F'| dm}.
\end{equation}
\end{theorem}

We say that a function $h: \mathbb{D} \to \mathbb{C}$ is {\em weakly almost invariant} under $F$ if for a.e.~every backward orbit ${\bf z} = (z_i)_{i=-\infty}^0 \in \widehat{\mathbb{D}}$, $\lim_{i \to -\infty} h(z_i)$ exists and defines a  function on the Riemann surface lamination:
$$
\widehat{h}({\bf z}) = \lim_{i \to -\infty} h(z_i).
$$

\begin{theorem}
\label{ergodic-theorem2}
Let $F$ be a centered inner function of finite Lyapunov exponent for which the geodesic flow on $\widehat{X}_{\lin}$ is ergodic.
Suppose $h: \mathbb{D} \to \mathbb{C}$ is a bounded weakly-almost invariant function that is uniformly continuous in the hyperbolic metric. Then for almost every $\zeta \in S^1$, we have
$$
\lim_{r \to 1} \frac{1}{|\log(1-r)|} \int_0^r h(s\zeta) \cdot \frac{ds}{1-s} = \fint_{\widehat X} \widehat{h} d\xi.
$$
In particular,
$$
\lim_{r \to 1} \frac{1}{2 \pi |\log(1-r)|} \int_{\mathbb{D}_r} h(z) \cdot \frac{dA(z)}{1-|z|} = \fint_{\widehat X} \widehat{h} d\xi.
$$
\end{theorem}

\begin{proof}
For simplicity, we first consider the case when $h: \mathbb{D} \to \mathbb{C}$ is {\em eventually invariant} under $F$, i.e.~there exists a $0 < \rho < 1$
such that
$$
h(F^{\circ n}(z)) = h(z), \qquad |F^{\circ n}(z)| > \rho.
$$
By the ergodic theorem, for $\xi$ a.e.~inverse orbit ${\bf z} \in \widehat{X}_{\lin}$, we have
\begin{equation}
\lim_{T \to \infty} \frac{1}{T} \int_0^T \widehat{h}(g_{-t}({\bf z})) dt = \fint_{\widehat X} \widehat{h} d\xi.
\end{equation}
By Theorem \ref{geodesic-foliation-theorem}(ii), for $\xi$ a.e.~${\bf z} \in \widehat{\mathbb{D}}_{\lin}$, $\{ g_{-t}({\bf z})_0 : t > 0 \}$ weakly shadows a radial ray $[0, \zeta_0({\bf z})]$.
Since $h$ is eventually invariant and $\{ g_{-t}({\bf z})_0 : t > 0 \}$ is eventually contained in the annulus $A(0; \rho, 1)$,
\begin{equation}
\label{eq:ergodic-theorem-general-case2}
\lim_{T \to \infty} \frac{1}{T} \int_0^T \widehat{h}(g_{-t}({\bf z})) dt =  \lim_{T \to \infty} \frac{1}{T} \int_0^T h(g_{-t}({\bf z})_0) dt.
\end{equation}
By the weak shadowing and the uniform continuity of $h$ in the hyperbolic metric,
\begin{equation}
\label{eq:meow}
\lim_{r \to 1} \frac{1}{|\log(1-r)|} \int_0^r h(s \cdot \zeta_0({\bf z})) \cdot \frac{ds}{1-s} = \fint_{\widehat X} \widehat{h} d\xi.
\end{equation}
According to Theorem \ref{geodesic-foliation-theorem}(iv), endpoints $\zeta({\bf z})$ of inverse orbits ${\bf z} \in \widehat{\mathbb{D}}_{\lin}$ satisfying (\ref{eq:meow}) cover the solenoid up to a $\widehat{m}$ measure zero set.
Projecting onto the $0$-th coordinate, we see that (\ref{eq:meow}) holds for $m$-a.e.~$\zeta \in S^1$.

We now turn to the general case when $h$ is only a weakly almost invariant function. The missing step is to show that (\ref{eq:ergodic-theorem-general-case2}) holds for $\xi$ almost every inverse orbit ${\bf z} \in \widehat{\mathbb{D}}_{\lin}$.

Given $\varepsilon > 0$ and $0 < \rho < 1$, let $E(\varepsilon, \rho) \subset \widehat{X}_{\lin}$ be the {\em complement} of the set of the inverse orbits ${\bf z} = (z_n)_{n=-\infty}^\infty$ for which
$$
|h(z_n) - \widehat{h}({\bf z})| < \varepsilon,
$$
for all $n \in \mathbb{Z}$ with $|z_n| > \rho$. By the definition of a weakly almost invariant function, for any fixed $\varepsilon > 0$,
$\xi(E(\varepsilon, \rho)) \to 0$ as $\rho \to 1$. We may therefore choose $\rho = \rho(\varepsilon)$ so that $\xi(E(\varepsilon, \rho)) < \varepsilon$.

By the ergodic theorem, a generic backward trajectory $\{ g_{-t}({\bf z}): t > 0 \}$ spends little time in $E(\varepsilon, \rho)$, i.e.~
$$
 \lim_{T \to \infty} \frac{1}{T}  \int_0^T \chi_{E(\varepsilon, \rho)}(g_{-t}({\bf z})) \, dt < \varepsilon.
$$
As $\{ g_{-t}({\bf z})_0: t > 0 \}$ is eventually contained in the annulus $A(0; \rho, 1)$, the difference
$$
\limsup_{T \to \infty} \frac{1}{T} \int_0^T \Bigl \{ \widehat{h}(g_{-t}({\bf z})) -  h(g_{-t}({\bf z})_0)  \Bigr \} dt \lesssim \varepsilon + \varepsilon \| h \|_\infty,
$$
which can be made arbitrarily small by requesting that $\varepsilon > 0$ is small, thereby justifying (\ref{eq:ergodic-theorem-general-case2}).
 \end{proof}
  
\subsection{A weakly almost invariant function}

To prove Theorems  \ref{main-thm-2a}, we will use a slight modification $h_{\nice}$ of the almost invariant function $h_{\smooth}$  from Section \ref{sec:1c-orbit-counting}, which was constructed by first defining $h_{\smooth}$ on a box $\square = \square(z, \delta)$ and then extending it to the repeated pre-images of $\square = \square(z, \delta)$ by invariance.

On the box $\square = \square(z, \delta)$, we set $h_{\nice} = h_{\smooth}$. Let $w$ be a repeated pre-image of $z$, i.e.~$F^{\circ n}(w) = z$ for some $n \ge 0$. Recall that $w$ is an $\varepsilon$-(linear good) pre-image if $e^\gamma(1 - |z|) < \varepsilon$ and 
$$
\widehat{\delta}(w, z) \, := \, \sum_{i=0}^n \delta(F^{\circ i}(w)) \, \le \, \varepsilon.
$$
When $\varepsilon > 0$ is sufficiently small, the connected component $$\square_w = F^{-1}(\square(z,\delta))$$ containing $w$ is a topological
disk which has roughly the same hyperbolic size and shape as $\square$.
On each such good box $\square_w$, we define $h_{\nice}$ by invariance. Outside the good boxes, we set $h_{\nice}$ to be zero.

In view of Theorem \ref{linear-decomposition-theorem}, $h_{\nice}$ is a weakly almost invariant function on the unit disk. Recall from Section  \ref{sec:1c-orbit-counting} that $h_{\nice} = h_{\smooth}$ was chosen to be uniformly continuous in the hyperbolic metric on $\square$. By the Schwarz lemma, $h_{\nice}$ is  uniformly continuous in the hyperbolic metric on $\mathbb{D}$. We denote its natural extension to the Riemann surface lamination by $\widehat{h}_{\nice}$.

The proof of Theorems \ref{main-thm-2a}  is nearly the same as that of Theorem \ref{main-thm-a}. We therefore point out the differences: In Step 1, we assume that $z \in A(0; 1 - \varepsilon, 1)$ is an $\varepsilon$-(linear nice) point and we show that 
\begin{equation}
\label{eq:z-close-to-the-unit-circle2}
\frac{1}{R} \sum_{\substack{F^{n}(w) = z, \, n \ge 0 \\ w \in B_{\hyp}(0,R), \, \good}} e^{-d_{\mathbb{D}}(0, w)} \, \sim_{\varepsilon,R} \, \frac{1}{2} \log \frac{1}{|z|} \cdot \frac{1}{\int_{\partial \mathbb{D}} \log |F'| dm},
\end{equation}
where we only count the number of $\varepsilon$-(linear good) pre-images. Steps 2 and 3 proceed as before for $\varepsilon$-(linearly decomposable) points, i.e.~points satisfying (\ref{eq:linear-decomposition-theorem}).

\part{Parabolic Inner Functions}
\label{part:parabolic}

By a {\em parabolic} inner function, we mean an inner function $F$ whose Denjoy-Wolff fixed point $p \in \partial \mathbb{D}$ with
$F'(p) := \lim_{r \to 1} F'(rp) = 1$.

We view parabolic inner functions as holomorphic self-maps of the upper half-plane, with the parabolic fixed point at infinity.
In this case, Lebesgue measure $\ell$ on the real line is invariant, e.g.~see \cite{doering-mane}. We say that a parabolic inner function $F: \mathbb{H} \to \mathbb{H}$ has {\em finite Lyapunov exponent} if
$$
\chi_{\ell} = \int_{\mathbb{R}} \log |F'(x)| d\ell < \infty.
$$

 By Julia's lemma, for any point $z_0 \in \mathbb{H}$, the imaginary parts $\{ \im F^{\circ n}(z_0) \}$ are increasing. We say that $F$ has {\em finite height} if  $\{ \im F^{\circ n}(z_0) \}$ are uniformly bounded and {\em infinite height}\/ if $\im F^{\circ n}(z_0) \to \infty$. In view of the Schwarz lemma, this definition is independent of the choice of the starting point $z_0 \in \mathbb{H}$. 

In this final part of the paper, we discuss orbit counting theorems for parabolic inner functions of infinite height. 
 As the proofs are essentially the same, we only give a brief description of the results and leave the details to the reader. 

\section{Statements of Results}
\label{sec:parabolic}

For a bounded interval $I \subset \mathbb{R}$ and a real number $R>0$, consider the counting function 
$$
\mathcal N_I(z, R) = \# \bigl \{ w \in I \times [e^{-R}, 1] : F^{\circ n}(w) = z \text{ for some }n \ge 0 \bigr \}.
$$

\begin{theorem}
\label{main-thm3}
Let $F: \mathbb{H} \to \mathbb{H}$ be an infinite height parabolic inner function of finite Lyapunov exponent. Suppose $z \in \mathbb{H}$ lies outside a set of zero measure. Then,
$$
\frac{1}{R} \int_0^R \frac{\mathcal N_I(z, S)}{e^S} dS \, \sim \,  |I| \cdot \frac{1}{\int_{\mathbb{R}} \log |F'| d\ell}
$$
as $R \to \infty$.
\end{theorem}

When a parabolic inner function $F: \mathbb{D} \to \mathbb{D}$ is holomorphic in a neighbourhood of the Denjoy-Wolff point $p \in \partial \mathbb{D}$, we can classify it as {\em singly parabolic} or  {\em doubly parabolic} depending on whether the Taylor expansion is
$$
F(z) = p + (z-p) + a_2(z-p)^2 + \dots, \qquad a_2 \ne 0
$$
or
$$
F(z) = p + (z-p) + a_3(z-p)^3 + \dots, \qquad a_3 \ne 0.
$$ 
Singly and doubly parabolic inner functions on the upper half-plane are defined by conjugating with a M\"obius transformation that takes $\mathbb{D}$ to $\mathbb{H}$. For example, $z \to z - 1/z + T$ is doubly-parabolic for $T = 0$, while singly-parabolic for $T \in \mathbb{R} \setminus \{0\}$. Singly parabolic functions have finite height, while doubly parabolic functions have infinite height.

\begin{theorem}
\label{main-thm4}
Let $F: \mathbb{H} \to \mathbb{H}$ be a doubly-parabolic one component inner function of finite Lyapunov exponent. For all $z \in \mathbb{H}$ lying outside a countable set, we have
$$
\mathcal N_I(z, R) \, \sim \, |I| \cdot \frac{1}{\int_{\mathbb{R}} \log |F'| d\ell},
$$
as $R \to \infty$.
\end{theorem}

\subsection{Background on parabolic inner functions}

In the upper half-plane, Lemmas \ref{sum-of-heights}, \ref{derivative-circle} and \ref{ac-lemma} read as follows:
\begin{lemma}
\label{sum-of-heights2}
Suppose $F$ is a parabolic inner function with the parabolic fixed point at infinity. For a {\em non-exceptional} point $z \in \mathbb{H}$,
\begin{equation}
\im z = \sum_{F(w)=z} \im w.
\end{equation}
\end{lemma}

An inner function viewed as self-mapping of the upper half-plane can be expressed as
$$
F(z) = \alpha z + \beta + \int_{\mathbb{R}} \frac{1 + zw}{w-z} d\mu(w),
$$
for some constants $\alpha > 0$, $\beta \in \mathbb{R}$ and a finite positive singular measure $\mu$ on the real line, e.g., see \cite{tsuji}.
Differentiating, we get
\begin{align*}
F'(z) & = \alpha+ \int_{\mathbb{R}} \frac{w(w-z)+(1+wz)}{(w-z)^2} \, d\mu(w), \\
& = \alpha+ \int_{\mathbb{R}} \frac{w^2+1}{(w-z)^2} \, d\mu(w).
\end{align*}

Since $\alpha = \lim_{t \to \infty} F'(it)$, an inner function has a parabolic fixed point at infinity if and only if $\alpha = 1$. The following two lemmas are straightforward consequences of the above formula:

\begin{lemma}
\label{derivative-on-the-real-line}
If $F$ is a parabolic inner function with the parabolic fixed point at infinity, then for a bounded interval $J$ in the real line, there exists a constant $c_J > 1$ such that $F'(\zeta) > c_J$ for all $\zeta \in J$.
\end{lemma}

\begin{lemma}
If $F(z)$ is an inner function, viewed as a map of the upper half-plane to itself, then
\begin{equation}
|F'(x+iy)| \le |F'(x)|
\end{equation}
for all $x+iy \in \mathbb{H}$.
\end{lemma}

\subsection{Riemann surface laminations}

For a parabolic inner function $F$, we may form the space of backward orbits
$$
\widehat{\mathbb{H}} \, = \,  \lim_{\longleftarrow} \, ( F : \mathbb{H} \to \mathbb{H} )
\, = \,  \bigl \{ (z_i)_{i=-\infty}^0 : F(z_i) = z_{i+1} \bigr \}.
$$  
The Riemann surface lamination is then defined as $\widehat{X} = \widehat{\mathbb{H}}/ \widehat F$.
In view of Lemma \ref{sum-of-heights2}, the natural
measure $d\xi$ on $\widehat{X}$ is now given by the formula
\begin{equation}
\label{eq:xi-def2}
\xi(\widehat{\mathscr B}) = \lim_{n \to \infty}  \int_{F^{-n}(\mathscr B)}  \frac{ |dz|^2} {\im z}.
\end{equation}
Adapting the proof of Theorem \ref{total-mass} to the current setting shows that
$$
\xi(\widehat{X}) = \int_{\mathbb{R}} \log|F'(x)|d\ell.
$$

\begin{remark}
(i) The infinite height condition guarantees that every inverse orbit passes through a backward fundamental domain of the form $$F^{-1}(\mathbb{H}_t) \setminus \mathbb{H}_t,$$ where $\mathbb{H}_t = \{z \in \mathbb{H} : \im z > t \}$.

(ii) Without the infinite height condition, the Riemann surface lamination $\widehat{X}$ may not have finite volume. For instance, for the singly parabolic Blaschke product
 $z \to z - 1/z +T$ with $T \in \mathbb{R} \setminus \{0\}$, the volume of  $\widehat{X}$ is infinite, even though
 $$
 \int_{\mathbb{R}} \log \biggl ( 1 + \frac{1}{z^2} \biggr ) d\ell(z) = 2\pi.
 $$

(iii) By Lemma \ref{derivative-on-the-real-line}, a generic inverse orbit $(z_i)$ does not converge to infinity, and therefore $\im z_i \to 0$.
\end{remark}

As in Section \ref{sec:linear-structure}, one can show:
 
\begin{lemma}
For a  finite Lyapunov exponent inner function $F: \mathbb{H} \to \mathbb{H}$ with a parabolic fixed point at infinity,
$$
\int_{\mathbb{H}}\delta(x+iy) \cdot \frac{dxdy}{y} < \infty.
$$
\end{lemma}

The above lemma implies that iteration along a.e.~inverse orbit is essentially linear and therefore a.e.~leaf of $\widehat{X}$ is covered by $(\mathbb{H}, \infty)$, which allows one to define geodesic and horocyclic flows on $\widehat{X}$.

The following theorems are analogues of Theorems \ref{ergodic-theorem} and \ref{mixing-theorem} respectively:

\begin{theorem}
For an infinite height parabolic inner function $F: \mathbb{H} \to \mathbb{H}$ of finite Lyapunov exponent, the geodesic flow on $\widehat{X}$ is ergodic. In particular, if
 $h: \mathbb{H} \to \mathbb{C}$ is a bounded almost invariant function that is uniformly continuous in the hyperbolic metric, then for almost every $x \in \mathbb{R}$, we have
$$
\lim_{t \to 0} \frac{1}{|\log t|} \int_t^1 h(x+iy) \cdot \frac{dy}{y}  = \frac{1}{\int_{\mathbb{R}} \log|F'| d\ell} \int_{\widehat X} \widehat{h} d\xi.
$$
\end{theorem}

\begin{theorem}
\label{parabolic-mixing-theorem}
For a doubly parabolic one component inner function $F: \mathbb{H} \to \mathbb{H}$ of finite Lyapunov exponent, the geodesic flow on $\widehat{X}$ is mixing. In particular, if $h: \mathbb{H} \to \mathbb{C}$ is a bounded almost invariant function that is uniformly continuous in the hyperbolic metric and $I \subset \mathbb{R}$ is a bounded interval, then
$$
\lim_{y \to 0} \int_{I} h(x+iy) d\ell(x) = \frac{|I|}{\int_{\mathbb{R}} \log|F'| d\ell} \int_{\widehat X} \widehat{h} d\xi.
$$
\end{theorem}
Again, the proofs are similar to the case when the Denjoy-Wolff point is inside the disk. (To show the mixing of the geodesic flow, we use that for doubly parabolic one component inner functions, the multipliers of the repelling periodic orbits on the real line do not belong to a discrete subgroup of $\mathbb{R}^+$, for a proof, see \cite[Section 9.4]{inner-tdf}.)

\part{Appendices}
\appendix

\section{A Shadowing Lemma}
\label{sec:shadowing-property}

The following theorem roughly says that if you drive a car in the upper half-plane with the desire to reach the real axis, and you are able to steer the car for most of the time, then on average, your path will be close to a vertical geodesic:

\begin{theorem} 
\label{weak-shadowing-in-H}
Let $\gamma: [0, \infty) \to \mathbb{H}$ be a $C^1$ parametrized curve in the upper half-plane with $\| \gamma'(t) \|_{\hyp} \le 1$.
Suppose $[0,\infty) = \mathcal G \cup \mathcal B$ is partitioned into good and bad times such that at good times, $\gamma'(t) = v_{\downarrow} = - y \cdot \frac{\partial}{\partial y}$, while at bad times, $\gamma'(t)$ can point in any direction. 

{\em (i)} If the upper density of bad times
\begin{equation}
\label{eq:good-and-bad-times}
\limsup_{T \to \infty} \frac{| \{0 < t < T : t \in \mathcal B \} |}{T} = 0,
\end{equation}
then the limit $\zeta = \lim_{t \to \infty} \gamma(t)$ exists and lies on the real axis.

{\em (ii)} Furthermore, if $\overline{\gamma}(t)$ is the vertical geodesic to $\zeta$, then
\begin{equation}
\label{eq:average-distance}
\frac{1}{T} \int_0^T \min \bigl \{ 1, d_{\mathbb{H}}(\gamma(t), \overline{\gamma})  \bigr \} \, dt \to 0, \qquad \text{as }T \to \infty.
\end{equation}
\end{theorem}

\begin{remark}
The above the theorem remains true if during a good time, we allow $\gamma'(t)$ to be only approximately equal to $v_{\downarrow}$, rather than exactly equal: it is enough to require that $\| \gamma'(t) - v_\downarrow \| < c$ for some $c < 1/2$.
\end{remark}

The proof of Theorem \ref{weak-shadowing-in-H} is based on the following simple observation:

\begin{lemma}
\label{weak-shadowing-in-H2}
Suppose $\sigma \ge 0$ is a locally finite singular measure on $[0, \infty)$ such that $\sigma([0, T])/T \to 0$ as $T \to \infty$.
The function 
$$
\Delta_\infty(t) = \int_t^\infty e^{-(\tau-t)} d\sigma(\tau)
$$
is sub-linear:
$
\Delta_\infty(T)/T \to 0
$
as $T \to \infty$.
\end{lemma}

The above lemma easily follows from Fubini's theorem. In the proof below, we will also use the function
$$
\Delta_T(t) = \int_t^T e^{-(\tau-t)} d\sigma(\tau).
$$

\begin{proof}[Proof of Theorem \ref{weak-shadowing-in-H}]
{\em Step 1.} For clarity, we first examine the case when during a bad time, $\gamma'(t) = v_{\rightarrow} = y \cdot \frac{\partial}{\partial x}$.
Consider the map $q : [0, \infty) \to [0, \infty)$ which ``collapses'' the set of bad times:
$$
q(t) = | \{0 \le s \le t : s \notin \mathcal B \} |.
$$
and let $\sigma = q_* (\chi_{\mathcal B} \, d\ell)$ be the push-forward of the part of the Lebesgue measure supported on $\mathcal B$. 
By assumption (\ref{eq:good-and-bad-times}) on the bad set, we have
$$
\frac{q(t)}{t} \to 1
\qquad \text{and} \qquad 
\frac{\sigma([0, T])}{T} \to 0, \quad \text{as }T \to \infty.
$$

From the definitions, is clear that $\Delta_T(q(t))$, with $0 < t \le T < \infty$, is the  hyperbolic length of the horizontal segment between $\gamma(t)$ and the vertical geodesic $\overline{\gamma}_T$ which passes through  $\gamma(T)$. Lemma \ref{weak-shadowing-in-H2} prevents the geodesic
$\overline{\gamma}_T$ from moving too much, so it converges as $T \to \infty$. We denote the limiting vertical geodesic by $\overline{\gamma}$. Lemma \ref{weak-shadowing-in-H2} also shows that restricted to good times, the average distance from $\gamma(t)$ to $\overline{\gamma}$ is small.

\medskip

{\em Step 2.} We now assume that during a bad time $$\gamma'(t) \, = \, v_{\uparrow} + v_{\rightarrow} \, = \, y \cdot \biggl \{ \frac{\partial}{\partial x} + \frac{\partial}{\partial y} \biggr \},$$
which is {\em worse} than the worst case scenario allowed in Theorem \ref{weak-shadowing-in-H}.
Let $\mathcal B^* \supset \mathcal B$ be the set of $s > 0$ for which there exists $t > s$ so that
$$
\bigl |[s,t] \cap \mathcal B \bigr |  \ge \frac{1}{3} \cdot |t - s|.
$$
In view of the Hardy-Littlewood Maximal Theorem, 
$$\bigl |[0,T] \cap \mathcal B^* \bigr | \le C \, \bigl |[0,T] \cap \mathcal B \bigr |, \qquad \text{for some }C > 0,$$ and therefore,
\begin{equation*}
\frac{| \{0 < t < T : t \in \mathcal B^* \} |}{T} \to 0.
\end{equation*}
This time, we define
$$
q(t) = | \{0 \le s \le t : s \notin \mathcal B^* \} |
$$
and $\sigma = q_* (\chi_{\mathcal B^*} \, d\ell)$. Inspection shows that $\Delta_T(q(t))$ provides an upper bound for the  hyperbolic length of the horizontal segment between $\gamma(t)$ and the vertical geodesic $\overline{\gamma}_T$. The proof is completed by Lemma \ref{weak-shadowing-in-H2} as in Step 1.
\end{proof}

\section{A Criterion for Angular Derivatives}

In this appendix, we show the following theorem, answering a question posed in \cite{BKR}:

\begin{theorem}
\label{angular-derivatives}
A holomorphic self-map of the unit disk $F$ has a finite angular derivative at $\zeta \in \partial \mathbb{D}$ in the sense of Carath\'eodory if and only if
\begin{equation}
\label{eq:mu-integral}
\int_0^\zeta \mu(z) \, d\rho \, = \, \int_0^\zeta \biggl  (1 - \frac{(1-|z|^2) |F'(z)|}{1-|F(z)|^2} \biggr ) \frac{2 |dz|}{1-|z|^2} \,<\, \infty.
\end{equation}
\end{theorem}

By composing with a M\"obius transformation, we may assume that $F(0) = 0$. By the Schwarz lemma, the function
$$
L(r) = \bigl \{ d_{\mathbb{D}}(0,r\zeta) -  d_{\mathbb{D}}(0,F(r\zeta))  \bigr \}, \qquad 0 < r < 1,
 $$
 is increasing. The limit
 $$
 \lim_{r \to 1} L(r) < \infty
 $$
is finite if and only if $F$ has an angular derivative at $\zeta$, in which case, $$ \lim_{r \to 1} L(r) = \log |F'(\zeta)|.$$
In other words, $F$ possesses an angular derivative at $\zeta$ if when moving from $0$ to $\zeta$ along the radial geodesic ray  $\gamma = [0, \zeta)$ at unit hyperbolic speed, the
image point efficiently moves toward the unit circle. Expressed infinitesimally, this says that $F$ has a finite angular derivative at $\zeta \in \partial \mathbb{D}$ if and only if
\begin{equation}
\label{eq:nu-integral}
\int_0^\zeta \eta(z) \, d\rho < \infty.
\end{equation}
The main difficulty in proving Theorem \ref{angular-derivatives} is replacing the radial inefficiency $\eta$ with the M\"obius distortion $\mu$.

\begin{proof}[Proof of Theorem \ref{angular-derivatives}]
Since $\mu \le \eta \le \mu+\alpha$, it is enough to show that
$$
\int_0^\zeta \mu(z) \, d\rho < \infty \qquad \implies \qquad
\int_0^\zeta \alpha(z) \, d\rho < \infty.
$$

{\em Step 1.} A compactness argument shows that for every $\varepsilon > 0$, there is a $\delta > 0$ so that if $\mu(z) < \delta$ then $\mu(w) < \varepsilon$ for all $w \in B_{\hyp}(z,1)$.

As a result, the M\"obius distortion $\mu(r\zeta) \to 0$ as $r \to 1$. Lemma \ref{geodesic-curvature-and-mobius-distortion} tells us that the geodesic curvature
$$
\kappa_{F(\gamma)}(F(r\zeta)) \to 0, \qquad r \to 1.
$$
Therefore, by Lemma \ref{quasigeodesic-property-of-hyperbolic-space}, $F(\gamma)$ lies within a bounded hyperbolic distance of the geodesic ray $[0, F(\zeta))$.
In particular, this shows that $F$ possesses a radial boundary value at $\zeta$ somewhere on the unit circle.

\medskip

{\em Step 2.}
By Lemma \ref{geodesic-curvature-and-mobius-distortion}, the total geodesic curvature of $F(\gamma)$ is finite:
$$
\int_{0}^\zeta \kappa_{F(\gamma)}(F(z)) \, d\rho < \infty.
$$
Since $F(\gamma)$ lies within a bounded hyperbolic distance of the geodesic ray $[0, F(\zeta))$, there is a sequence of $r_n$'s tending to 1 so that
$\alpha_{F(r_n \zeta)} < 2\pi/3$. (It is not possible for $F(\gamma)$ to approach the unit circle if the tangent vector always points away from the unit circle.)

Therefore, there exists an $0 < r_n < 1$ so that
$$
\alpha_{F(r_n \zeta)} < 2\pi/3 \qquad \text{and} \qquad
\int_{r_n\zeta}^\zeta \kappa_{F(\gamma)}(F(z)) \, d\rho < 0.1.
$$
Lemma \ref{inclination-bound} tells us
$$
\int_{r_n \zeta}^\zeta \alpha(z) \, d\rho = O(1),
$$
which is what we wanted to show.
\end{proof}

\section{Integrating over Leaves}

\label{sec:integrating-over-leaves}

In this appendix, we prove Theorem \ref{novel-formulas-for-measure}, which describes the measures $\xi$ and $\widehat{m}$ in terms of integration along leaves, similar to McMullen's original definitions of these measures given in Section \ref{sec:natural-measures}.

 \subsection{\texorpdfstring{The case of $\widehat{\mathbb{D}}$}{The case of the lamination}}

 We define a  $\sigma$-finite measure  $\xi_{\leaf}$ on the solenoid $\widehat{X}$ so that its restriction to any ``chart'' of the form $\widehat{\mathscr B}_{\Lgood} \subset \widehat{X}$ is given by
$$
\xi_{\leaf}(A) =  \int_{T_{\Lgood}(z)} \biggl \{  \int_{A_{\bf z}^*} \frac{dA(w)}{\im w} \biggr \} dc_z,
$$
while the set of points not contained in any chart have $\xi_{\leaf}$ measure zero. After lifting to $\widehat{\mathbb{D}}$, we obtain an $\widehat{F}$-invariant measure on $\widehat{\mathbb{D}}$, which we also denote by $\xi_{\leaf}$.  Our objective is to show that  $\xi = \xi_{\leaf}$\,:

\begin{theorem}
\label{xi-equality}
The measures $\xi$ and $\xi_{\leaf}$ on $\widehat{\mathbb{D}}$ are equal.
\end{theorem}

We begin by checking that the measure $\xi_{\leaf}$ is well-defined:

\begin{lemma}
\label{disk-well-defined}
If 
$\mathscr B' = B_{\hyp}(z', \gamma)$ is another ball of hyperbolic radius $\gamma$ which intersects $\mathscr B$ and
$A \subset \widehat{\mathscr B} \cap \widehat{\mathscr B'}$ then
$$
 \int_{T_{\Lgood}(z)} \biggl \{  \int_{A^*_{\bf z}} \frac{dA(w)}{\im w} \biggr \} dc_z =  \int_{T_{\Lgood}(z')}  \biggl \{  \int_{A^*_{\bf z}} \frac{dA(w)}{\im w} \biggr \} dc_{z'}.
 $$
\end{lemma}

\begin{proof}
Given an inverse orbit ${\bf z} \in T(z)$, we can select an inverse orbit ${\bf z'} \in T(z')$ which follows ${\bf z}$ by using the same inverse branches. As the dynamics is asymptotically linear, the limit
$$
\rho_{\bf z, \bf z'} = \lim_{n \to \infty} \frac{1-|z'_{-n}|}{1-|z_{-n}|}
$$
exists. Inspection shows that $\rho_{\bf z, \bf z'} = dc_{z'}/dc_{z}$ is just the Radon-Nikodym derivative
of the transverse measures $c_z$ and $c'_z$.

Recall from Section \ref{sec:mobius-linear-laminations} that when we define the slice $A^*_{\bf z} \subset \mathbb{H}$, we rescale by a M\"obius transformation so that $z_{-n} \in \mathbb{D}$ maps to $i \in \mathbb{H}$, while when we define the slice $A^*_{\bf z'} \subset \mathbb{H}$, we rescale so that $z'_{-n} \in \mathbb{D}$ maps to $i \in \mathbb{H}$. Consequently, when changing from $z$ to $z'$, the integrand $ \int_{A^*_{\bf z}} \frac{dA(w)}{\im w} $ decreases by the factor $\rho_{\bf z, \bf z'}$, compensating for the Radon-Nikodym derivative. 
As a result, the expression for $\xi_{\leaf}(A)$ remains unchanged.
\end{proof}

\begin{lemma}
\label{novel-lamination1}
The measure $\xi_{\leaf}$ is absolutely continuous with respect to $\xi$.
\end{lemma}

\begin{proof}
To prove the lemma, it is enough to show that $\xi_{\leaf}(\widehat{A}) = 0$ for any Borel set $A \subset \mathbb{D}$ with $\xi(\widehat{A}) = 0$, as sets of this form generate the $\sigma$-algebra of Borel subsets of $\widehat{\mathbb{D}}$. From the definition of the measure $\xi$ given in Section \ref{sec:area}, it is easy to see that one has ``$\xi(\widehat{A}) = 0$'' if and only if ``$A$ has 2-dimensional Lebesgue measure zero.''
As a result, we need to show that $\xi_{\leaf}(\widehat{A}) = 0$ for any measurable set $A \subset \mathbb{D}$ with 2-dimensional Lebesgue measure zero. 

For this purpose, consider a chart $\widehat{\mathscr B}_{\Lgood}$ where $\mathscr B = B_{\hyp}(z, \gamma)$. As the slice $(\widehat{A} \cap  \widehat{\mathscr B}_{\Lgood})^*_{\bf z} \subset \mathbb{H}$ along any inverse orbit ${\bf z} \in T(z)$ also has zero 2-dimensional Lebesgue measure,
$\xi_{\leaf}(\widehat{A} \cap  \widehat{\mathscr B}_{\Lgood}) = 0.$
 Since the chart $\widehat{\mathscr B}_{\Lgood}$ was arbitrary, $\xi_{\leaf}(\widehat{A}) = 0$ as desired.
 \end{proof}

For a measurable set $A$ contained in a ball $\mathscr B = B_{\hyp}(z, \gamma)$, we write 
$$
\widehat{A}_{\Lgood} = \widehat{A} \cap \widehat{\mathscr B}_{\Lgood}.
$$ 
Perhaps, the main difficulty in showing that $\xi = \xi_{\leaf}$ is that the measure $\xi$ was defined in terms of
the ``full'' cylinders $\widehat{A}$ while the measure $\xi_{\leaf}$ is given in terms of the ``partial'' cylinders $\widehat{A}_{\Lgood}$.

In the following two lemmas, we evaluate $\xi(\widehat{A}_{\Lgood})$ and $\xi_{\leaf}(\widehat{A}_{\Lgood})$ up to multiplicative error $\varepsilon$. As before, we use $A \sim_\varepsilon B$ to denote that
$$
1-C\varepsilon \le \, A/B  \le \, 1+ C\varepsilon,
$$
for some constant $C > 0$ depending only on the inner function $F$.

\begin{lemma}
\label{xi-good}

We have
\begin{align}
\label{eq:xi-good1}
\xi(\widehat{A}_{\Lgood}) & =  \lim_{n \to \infty} \sum_{\substack{F^{\circ n}(w) = z \\ w \Lgood}} \int_{A_w}  \log \frac{1}{|z|} \, dA_{\hyp} \\
\label{eq:xi-good2}
& \sim_\varepsilon \, \overline{c}_z(T_{\Lgood}(z)) \int_A \log \frac{1}{|z|} \, dA_{\hyp}.
\end{align}
\end{lemma}

\begin{proof}
{\em Step 1.} We may write $F^{-j}(A) = G_j \sqcup B_j$, where $G_j$ is the union of the $\varepsilon$-(linear good) pre-images of $A$ of generation $j$ and $B_j$ be the union of the ``bad'' pre-images. Then,
 $$
 \widehat{A}_{\Lgood} = \widehat{A} \setminus \bigsqcup_{j=1}^\infty \widehat{B}_{j},
 $$
where we are slightly abusing notation by viewing
 $\widehat{B}_{j}$ as a subset of $\widehat{A}$. (We should really be writing $\widehat{F}^{\circ j}(\widehat{B}_j)$ in place of $\widehat{B}_j$.) Consequently,
 $$
 \widehat{\xi}( \widehat{A}_{\Lgood}) = \xi( \widehat{A} ) - \sum_{j=1}^\infty   \xi(\widehat{B}_{j}).
 $$

{\em Step 2.}
From the definition of the measure $\xi$ on the cylindrical sets $\widehat{A}$ and $\widehat{B}_j$ and Lemma \ref{monotonicity-lemma},
it follows that for any $n \in \mathbb{N}$, we have
\begin{equation}
\label{eq:xi-good3}
 \xi( \widehat{A} ) - \sum_{j=1}^n  \xi(\widehat{B}_{j}) \ge \sum_{\substack{F^{\circ n}(w) = z \\ w \Lgood}} \int_{A_w}  \log \frac{1}{|z|} \, dA_{\hyp},
\end{equation}
where $A_w = F^{-n}(A) \cap \mathscr B_w$ ranges over the $\varepsilon$-(linear good) pre-images of $A$ of generation $n$. 
In Section 11, we saw that the error
$$
\err(n, \widehat{A}) \, := \, \xi(\widehat{A}) -  \int_{F^{-n}(A)} \log \frac{1}{|z|} \, dA_{\hyp}.
$$
decreases to 0 as $n \to \infty$. As $\err(n, \widehat{G}_n) \le \err(n, \widehat{A})$,
 \begin{equation}
\label{eq:xi-good4}
 \xi( \widehat{A} ) - \sum_{j=1}^{n}  \xi(\widehat{B}_{j}) 
  \le  \err(n, \widehat{A}) +  \sum_{\substack{F^{\circ n}(w) = z \\ w \Lgood}} \int_{A_w}  \log \frac{1}{|z|} \, dA_{\hyp}.
\end{equation}
Taking $n \to \infty$ in (\ref{eq:xi-good3}) and (\ref{eq:xi-good4}), we obtain (\ref{eq:xi-good1}).

\medskip

{\em Step 3.}
 For $j \ge 1$, let $T^{(j)}_{\Lgood}(z) \subset T(z)$ denote the set of inverse orbits ${\bf w} \in T(z)$ which are $\varepsilon$-(linear good) for the first $j$ steps, i.e.~$\widehat{\delta}(w_{-j}, z) \le \varepsilon$.
Since
$$
T_{\Lgood}(z) = \bigcap_{n=1}^\infty T^{(n)}_{\Lgood}(z)
$$
is a decreasing intersection,
$\overline{c}_z(T^{(n)}_{\Lgood}(z))$
decreases
to $\overline{c}_z(T_{\Lgood}(z))$. With this in mind, (\ref{eq:xi-good2}) follows from (\ref{eq:xi-good1}) and $\varepsilon$-linearity.
\end{proof}

\begin{lemma}
\label{novel-lamination2}
For a measurable set $A$ contained in $\mathscr B = B_{\hyp}(z, \gamma)$, 
$$
 \xi_{\leaf}(\widehat{A}_{\Lgood}) \, \sim_\varepsilon \, \overline{c}_z(T_{\Lgood}(z)) \int_A \log \frac{1}{|z|} \, dA_{\hyp}.
$$
\end{lemma}

\begin{proof}
The lemma follows from the definition of the measure $\xi_{\leaf}$ and $\varepsilon$-linearity.
\end{proof}

With help of Theorem \ref{linear-decomposition-theorem}, one may express a cylinder set as a countable union of partial cylinders:

\begin{lemma}
\label{novel-lamination3}
For any measurable set $A$ in the unit disk and $\varepsilon > 0$, there exists countably many disjoint partial cylinders
$\widehat{A}_{k,\, \Lgood}$ which cover $\widehat{A}$ up to a set of $\xi$ measure zero:
$$
\widehat{A} = \bigsqcup_k \widehat{A}_{k,\, \Lgood} \sqcup N.
$$
\end{lemma}

We are now ready to show that the measures $\xi$ and $\xi_{\leaf}$ are equal:

\begin{proof}[Proof of Theorem \ref{xi-equality}]
To show that the measures $\xi$ and $\xi_{\leaf}$ are equal, it is enough to show that they agree on sets of the form
$\{ \widehat{A}: A \subset \mathbb{D} \text{ Borel} \}$ as these generate the Borel $\sigma$-algebra of Borel subsets of $\widehat{\mathbb{D}}$.
For a cylinder $\widehat{A} \subset \widehat{\mathbb{D}}$, examine the decomposition given by Lemma \ref{novel-lamination3}.
As $\xi_{\leaf}$ is absolutely continuous with respect to $\xi$, we also have $\xi_{\leaf}(N) = 0$. Lemmas \ref{xi-good} and \ref{novel-lamination2} imply that
$$\xi(\widehat{A}_{k,\, \Lgood}) \sim_\varepsilon \xi_{\leaf}(\widehat{A}_{k,\, \Lgood})$$ for any $k$. Summing over $k$ shows that
$\xi(\widehat{A}) \sim_\varepsilon \xi_{\leaf}(\widehat{A})$.
Since $\varepsilon > 0$ was arbitrary,  $\xi(\widehat{A}) = \xi_{\leaf}(\widehat{A})$ as desired.
\end{proof}
  
 \subsection{\texorpdfstring{The case of $\widehat{S^1}$}{The case of the solenoid}}
 
 We define a measure $\widehat{m}_{\leaf}$ on the solenoid $\widehat{S^1}$ so that its restriction to any ``chart'' $\zeta(\widehat{\mathscr B}_{\Lgood}) \subset \widehat{S^1}$ is given by
$$
\widehat{m}_{\leaf}(E) = \int_{T_{\Lgood}(z)} \ell(E^*_{\bf z}) dc_z,
$$
while the set of points in the solenoid which are not contained in any chart have $\widehat{m}_{\leaf}$ measure zero. 
As in the case of $\xi_{\leaf}$ considered previously, $\widehat{m}_{\leaf}$ is a $\sigma$-finite $\widehat{F}$-invariant measure. Our objective is to show:

\begin{theorem}
\label{m-equality}
The measures $\widehat{m}_{\leaf}$ and $\widehat{m}$ on $\widehat{S^1}$ are equal.
\end{theorem}

We begin by noticing:

\begin{lemma}
The measure  $\widehat{m}_{\leaf}$ is absolutely continuous with respect to $\widehat{m}$. 
 \end{lemma}
 
 The proof below uses L\"owner's lemma which says that if $\varphi: (\mathbb{D}, a) \to (\mathbb{D}, b)$ is a holomorphic self-map of the unit disk then for any measurable set $E \subset S^1$, $\omega_a(\varphi^{-1}(E)) \le \omega_b(E)$, where $\omega_a$ and $\omega_b$ are harmonic measures on the unit circle as viewed from $a$ and $b$ respectively. Evidently, L\"owner's lemma also applies to maps between arbitrary simply-connected domains.
 
 \begin{proof}
Let $E \subset S^1$ be a Borel set with $m(E) = 0$. Consider a chart $\zeta(\widehat{\mathscr B}_{\Lgood})$ where $\mathscr B = B_{\hyp}(z, \gamma)$.
For any inverse orbit ${\bf z} \in T_{\Lgood}(z)$, we can apply L\"owner's lemma to the map $F_{{\bf z}, 0}: (\mathbb{H}, i) \to (\mathbb{D}, z)$ to conclude that 
$$
\ell
\bigl ((\widehat{E} \cap \zeta(\widehat{\mathscr B}_{\Lgood}))_{\bf z}^* \bigr ) = 0.
$$
As the chart $\zeta(\widehat{\mathscr B}_{\Lgood})$ and inverse orbit ${\bf z} \in T_{\Lgood}(z)$ were arbitrary, we have
  $\widehat{m}_{\leaf}(\widehat{E}) = 0$.
\end{proof}

Since $\widehat{m}_{\leaf}$ is $\widehat{F}$-invariant and $\widehat{m}$ is ergodic, the above lemma tells us that:

\begin{corollary}
The measure $\widehat{m}_{\leaf}$ is finite. In fact, $\widehat{m}_{\leaf} = c \cdot \widehat{m}$ for some $c \ge 0$.
\end{corollary}

To complete the proof of Theorem \ref{m-equality}, it remains to show that $c = 1$. Unfortunately, we do not have a simple proof of this fact and the argument below is somewhat involved.

 \paragraph*{Step 1.}

We say that a trajectory of the geodesic flow $\{ g_{t}({\bf z}): t \in \mathbb{R} \}$ is {\em generic} if 
$$
 \lim_{t \to \infty} \frac{1}{t} \int_0^t \widehat{\delta} \bigl ( \widehat{F}^{\circ n} [ g_{-s}({\bf w}) ] \bigr ) ds = 0, \qquad \text{for any }n \in \mathbb{Z}.
$$
Let $\mathcal G_0$ be the set of generic trajectories. Recall that in Section \ref{sec:trajectories-land-on-the-solenoid}, we used the ergodic theorem to show that $\mathcal G_0$ foliates $\widehat{\mathbb{D}}$ up to $\xi$ measure zero. We also saw that under the backward geodesic flow, a generic trajectory lands on the solenoid.

We define the measure $\widehat{m}_{\gen}$ as the restriction of $\widehat{m}_{\leaf}$ to the set $\zeta(\mathcal G_0)$ of landing points of generic trajectories. Since $\mathcal G_0$ is $\widehat{F}$-invariant (by definition), so are
$\zeta(\mathcal G_0)$ and $\widehat{m}_{\gen}$. 
Notice that $\widehat{m}_{\leaf} - \widehat{m}_{\gen} \perp \widehat{m}_{\gen}$ as the two measures are supported on different sets: $\widehat{m}_{\gen}$ gives full mass to $\zeta(\mathcal G_0)$, while
$\widehat{m}_{\leaf} - \widehat{m}_{\gen}$ gives full mass to $\widehat{S^1} \setminus \zeta(\mathcal G_0)$.

\begin{lemma}
\label{m-star-probability-measure}
The measure $\widehat{m}_{\gen}$ is a probability measure.
\end{lemma}

Once we prove the above lemma, $c = 1$ follows almost immediately: As $\widehat{m}$ is ergodic and $\widehat{m}_{\gen} <\!\!< \widehat{m}$, the two measures must be equal: $\widehat{m} = \widehat{m}_{\gen}$.
As the difference $\widehat{m}_{\leaf} - \widehat{m}_{\gen} <\!\!< \widehat{m} = \widehat{m}_{\gen}$, it must be zero. Hence, $\widehat{m} = \widehat{m}_{\gen} = \widehat{m}_{\leaf}$ as desired.

\paragraph*{Step 2.}

For $0 < \varepsilon < 0.1$, we define $\mathcal A_{\varepsilon} \subset \widehat{\mathbb{D}}$ as the set of inverse orbits 
${\bf w} = (w_{-n})_{n=0}^\infty$ which satisfy the following three conditions:
\begin{enumerate}
\item $\widehat{\delta}({\bf w}) < \varepsilon.$
\item For any $t > 0$,
the hyperbolic distance 
$
d_{\mathbb{D}}(g_{-t}({\bf z})_{0}, 0) > d_{\mathbb{D}}(z_{0}, 0).
$
\item The geodesic trajectory passing through ${\bf w}$ is generic.
\end{enumerate}
For each $1-\varepsilon/e^\gamma < r < 1$, we define the auxiliary measure
$$
\widehat{m}_{r, \varepsilon} = \widehat{m}_{\leaf}|_{\zeta(\mathcal A_{r, \varepsilon})},
$$
where
$\mathcal A_{r, \varepsilon} = \mathcal A_\varepsilon \cap \{ |z| = r \}.$
From Condition 3, it is clear that
$$
\widehat{m}_{r, \varepsilon} \le \widehat{m}_{\gen} \le \widehat{m}_{\leaf}.
$$
Recall from Section \ref{sec:transversals} that the set of points $z \in \mathbb{D}$  for which $\overline{c}_z$ is not a probability measure has logarithmic capacity zero. In particular, the intersection with any circle $\{|z| = r \}$ has zero 1-dimensional Lebesgue measure.
The main difficulty towards proving Lemma \ref{m-star-probability-measure} is to show that the measures $\widehat{m}_{r, \varepsilon}$ exhaust $\widehat{m}_{\gen}$ as $r \to 1$\,:
\begin{lemma}
\label{most-inverse-orbits-lie-in-A}
For any $0 < \varepsilon <  0.1$,
\begin{equation}
\label{eq:goal}
\lim_{r \to 1} \int_{|z| = r} \overline{c}_z (\mathcal A_\varepsilon^c \cap T(z)) \, |dz| = 0.
\end{equation}
\end{lemma}

 We now explain how to derive  Lemma \ref{m-star-probability-measure} (and Theorem \ref{m-equality}) from Lemma \ref{most-inverse-orbits-lie-in-A}.
By Condition 2 above, for each non-exceptional $0 < r < 1$, $\zeta$ is injective on $\mathcal A_{r,\varepsilon}$. By $\varepsilon$-linearity, the mass of $\widehat{m}_{r, \varepsilon}$ is approximately
$$
\widehat{m}_{r, \varepsilon}(\widehat{S^1}) \, \sim_\varepsilon \, \frac{1}{2\pi} \int_{|z| = r} \overline{c}_z(\mathcal A_\varepsilon \cap T(z)) \, |dz|.
$$
Together with Lemma \ref{most-inverse-orbits-lie-in-A}, this implies that 
\begin{equation}
\label{eq:most-inverse-orbits-lie-in-A2}
\widehat{m}_{r, \varepsilon}(\widehat{S^1}) \sim_\varepsilon 1.
\end{equation}

Since any generic geodesic trajectory participates in ``density 1'' measures $\widehat{m}_{s, \varepsilon}$, i.e.~
$$
\frac{1}{|\log (1-r)|} \int_{0}^r \chi_{\mathcal A_{\varepsilon}}(g_s({\bf z})) \, \frac{ds}{s} \to 1, \qquad \text{as }r \to 1,
$$
we have:

\begin{lemma}
For any $0 < \varepsilon < 0.1$,
\begin{equation}
\label{eq:cesaro-representation}
\widehat{m}_{\gen} = \lim_{r \to 1} \frac{1}{|\log (1-r)|} \int_0^r \widehat{m}_{s, \varepsilon} \cdot \frac{ds}{s},
\end{equation}
in the sense of strong limits of measures. 
\end{lemma}

Combining (\ref{eq:most-inverse-orbits-lie-in-A2}) and (\ref{eq:cesaro-representation}), we see that $\widehat{m}_{\gen}$ is a probability measure.

\paragraph*{Step 3.}

By Lemmas \ref{delta-derivative} and \ref{cumulative-distortion}, there exists a universal constant $0 < \gamma_0 < \gamma$ so that if 
${\bf z} \in \widehat{\mathbb{D}}$ is an inverse orbit with $\widehat{\delta}({\bf z}) < 0.1$ then $\widehat{\delta}({\bf w}) < 2 \, \widehat{\delta}({\bf z}) < 0.2$ for any inverse orbit ${\bf w} \in \widehat{\mathscr B}_{\Lgood}$ which follows ${\bf z}$ with $d_{\mathbb{D}}(z_0, w_0) < \gamma_0$. In particular, 
\begin{equation}
\label{eq:onedle}
d_{\mathbb{D}}(g_{-t}({\bf z})_{-n}, 0) > d_{\mathbb{D}}(z_{-n}, 0), \qquad t \in (0, \gamma_0]
\end{equation}
and
\begin{equation}
\label{eq:twoodle}
d_{\mathbb{D}}(g_{-\gamma_0}({\bf z})_{-n}, 0) > d_{\mathbb{D}}(z_{-n}, 0) + 0.8 \, \gamma_0,
\end{equation}
for any $n \ge 0$.

We define the set $\widetilde{\mathcal A}_{\varepsilon} \subset \mathcal A_{\varepsilon} \subset \widehat{\mathbb{D}}$, where Condition 2 is replaced 
with a slightly stronger condition ($2+2'$), where we {\em additionally} require

\begin{enumerate}
\item[$2'.$] For any $t > \gamma_0$, we have $d_{\mathbb{D}}(0, g_{-t}({\bf w})_0) > d_{\mathbb{D}}(0, w_0) + \gamma_0/2.$
\end{enumerate}

In view of the buffer provided by $(2')$, we have:

\begin{lemma}
\label{stability-lemma}
Suppose $0 < \varepsilon < 0.05$. There exists $0 < \gamma_1 < \gamma_0$ so that if ${\bf z} \in \widetilde{\mathcal A}_\varepsilon$ then any generic orbit  ${\bf z'} \in \widehat{\mathscr B}_{\Lgood}$ which follows ${\bf z}$
with $d_{\mathbb{D}}(z_0, z'_0) < \gamma_1$ belongs to $\mathcal A_{2\varepsilon}$.
\end{lemma}

\paragraph*{Step 4.}

The following lemma says that from some point on, almost every inverse orbit belongs to $\widetilde{\mathcal A}_\varepsilon$\,:

\begin{lemma}
\label{high-preimages-A}
For $\xi$ a.e.~inverse orbit ${\bf w} \in \widehat{\mathbb{D}}$, there exists an $N({\bf w)} \ge 0$ such that
$\widehat{F}^{-n}({\bf w}) \in \widetilde{\mathcal A}_\varepsilon$ for all $n \ge N({\bf w})$. 
\end{lemma}

\begin{proof}
Recall that by Theorem \ref{nearly-linear}, for $\xi$ a.e.~inverse orbit, we have $\widehat{\delta}({\bf w}) < \infty$ and therefore,
$
\widehat{\delta}(\widehat{F}^{-n}({\bf w})) \to 0$ as $n \to \infty.
$
Consequently, for $n$ sufficiently large, $\widehat{\delta}(\widehat{F}^{-n}({\bf w})) < \varepsilon$ and Condition 1 holds. 

Condition 3 is also easy to check since $\xi$ a.e.~inverse orbit is generic and the property of an inverse orbit belonging to a generic trajectory is $\widehat{F}$-invariant by definition.
To verify Condition $(2+2')$, we examine three cases:

\begin{enumerate}
\item For $t \in (0, \gamma_0]$, Condition $2'$ for $\widehat{F}^{-n}({\bf w})$ follows from Condition 1 for $\widehat{F}^{-n}({\bf w})$ and (\ref{eq:onedle}).

\item By the definition of a generic trajectory, there exists a $T = T({\bf w}) > 0$ sufficiently large so that
$$
\frac{1}{t} \int_0^t \widehat{\delta}(g_{-s}({\bf w})) ds < 1/2, \qquad t > T.
$$
As a result, for $t > T$, we have
$$
d_{\mathbb{D}}(0, g_{-t}({\bf w})_{-n}) > d_{\mathbb{D}}(0, g_{0}({\bf w})_{-n}) + t/2.
$$

\item Finally, to handle the case when $t \in [\gamma_0, T]$, we use that
the sequence of functions
$$
\Delta_n(t) = \widehat{\delta} (\widehat{F}^{-n} [g_{-t}({\bf w})] )  = \sum_{k=n+1}^\infty \delta(g_{-t}({\bf w}))_{-k},
$$
decreases pointwise to 0.
\end{enumerate}
The proof is complete.
\end{proof}

For an inverse orbit ${\bf w} = (w_n)_{n = -\infty}^\infty \in \widehat{\mathbb{D}}$ and $0 < r < 1$, we write $w_r$ for the last point of the orbit that lies in the annulus $A(0; r,1)$, that is,
$w_r = w_{n(r)}$ where $n(r) \in \mathbb{Z}$ is the largest integer for which $w_{n(r)} \in A(0; r, 1)$.
One may interpret Lemma \ref{high-preimages-A} as saying that
\begin{equation}
\label{eq:goal2}
\int_{\widehat{X}} \chi_{\{w_r \in \widetilde{\mathcal A}^c_\varepsilon \}} d\xi({\bf w}) \to 0, \qquad \text{as }r \to 1.
\end{equation}
With the above preparations, we are now ready to prove Lemma \ref{most-inverse-orbits-lie-in-A}:

\begin{proof}[Proof of Lemma \ref{most-inverse-orbits-lie-in-A}]
Suppose that one could find a sequence of $r$'s tending to 1 so that
\begin{equation*}
\int_{|z| = r} \overline{c}_z (\mathcal A^c_{2\varepsilon} \cap T(z)) \, |dz| \ge \delta,
\end{equation*}
for some $\delta > 0$. By Lemma \ref{stability-lemma}, we would also have
$$
\int_{|z| = s} \overline{c}_z (\widetilde{\mathcal A}^c_\varepsilon \cap T(z)) \, |dz| \ge \delta,
$$
for any $0 < s < 1$ with $d_{\mathbb{D}}(r,s) < \gamma_1/2$. Consequently,
\begin{equation}
\label{eq:average-over-the-annulus}
\frac{1}{\gamma_0} \int_A  \overline{c}_z (\mathcal A^c_\varepsilon \cap T(z)) \cdot \frac{2 \, dA(z)}{1-|z|^2} \, \ge \, \delta,
\end{equation}
where
$A =  \{ z \in \mathbb{D} : d_{\mathbb{D}} (|z|, r) < \gamma_1/2 \}$ is an annulus of hyperbolic width $\gamma_1$.
Since we requested that $\gamma_1 < \gamma$,   the quotient map $\pi: \widehat{\mathbb{D}} \to \widehat{X}$ is injective on  $\widehat{A}$ and (\ref{eq:average-over-the-annulus}) contradicts  (\ref{eq:goal2}) if $r$ is sufficiently close to 1.
\end{proof}

\subsection*{Acknowledgements}
 The authors wish to thank Mikhail Lyubich for bringing the work of Glutsyuk to our attention.
This research was supported by the Israeli Science Foundation (grant no.~3134/21) and the Simons Foundation (grant no.~581668).

\bibliographystyle{amsplain}

\end{document}